\documentclass[a4paper,12pt]{article}
\usepackage[a4paper,centering,scale=.7]{geometry}
\newif\ifShowMarginPar

\ShowMarginParfalse 

\usepackage{amsmath,amssymb}
\usepackage{graphicx}
\usepackage{mathrsfs}

\usepackage{marginnote}

 \usepackage{pst-all}
\newcommand{\Z}{\mathbb Z}
\newcommand{\N}{\mathbb N}
\newcommand{\R}{\mathbb R}
\newcommand{\Zz}{{\mathbb Z}^2}

\renewcommand{\P}{\mathbb P}
\newcommand{\pn}{\pi^{(n)}}

\newcommand{\E}{\mathbb E}
\newcommand{\V}{\mathbb{V}\mathrm{ar}}
\newcommand{\bu}{\mathbf{u}}
\newcommand{\bv}{\mathbf{v}}
\newcommand{\bz}{\mathbf{z}}
\newcommand{\bw}{\mathbf{w}}
\newcommand{\bx}{\mathbf{x}}
\newcommand{\by}{\mathbf{y}}
\newcommand{\prob}{\xrightarrow{\mathbb{P}}}

\newcommand{\disteq}{\overset{d}{=}}
\newcommand{\n}{n_{\mathbf{u},\mathbf{v}}}
\newcommand{\h}{h^{\mathtt{IND}}}
\newcommand{\T}{\mathtt{T}^{\mathtt{IND}}}
\newcommand{\hh}{\bar{h}^{\mathtt{IND}}}

\newtheorem{theorem}{Theorem}[section]

\newtheorem{definition}[theorem]{Definition}

\newtheorem{lemma}[theorem]{Lemma}

\newtheorem{proposition}[theorem]{Proposition}
\newtheorem{remark}[theorem]{Remark}

\newenvironment{proof}[1][Proof]{\noindent\textbf{#1.} }{\ \rule{0.5em}{0.5em}}
\begin{document}

\title{A drainage network with dependence and the Brownian web}
\author{Azadeh Parvaneh\thanks{%
Postal address: Department of Statistics, Faculty of Mathematics and Statistics, University of Isfahan, Isfahan 81746-73441,
Iran.  Email: azadeh.parvaneh@sci.ui.ac.ir }~~~~~
Afshin Parvardeh\thanks{%
Postal address: Department of Statistics, Faculty of Mathematics and Statistics, University of Isfahan, Isfahan 81746-73441,
Iran. Email: a.parvardeh@sci.ui.ac.ir }~~~~~
Rahul Roy\thanks{%
Postal address: Indian Statistical Institute, 7 S. J. S. Sansanwal Marg, New Delhi 110016, India.
Email: rahul@isid.ac.in }}
\maketitle
\date{}

\begin{abstract}
\noindent We study a system of coalescing  random walks on the integer lattice $ \Z^{d} $ in which the walk is oriented in the $d$-th direction and follows certain specified rules. We first study the geometry of the paths and show that,  almost surely, the paths form a graph consisting of just one tree for dimensions $ d=2,3 $ and infinitely many disjoint trees for dimensions $ d\geq 4 $. Also, there is no bi-infinite path in the graph almost surely for $ d\geq 2 $. Subsequently, we prove that for $ d=2 $ the diffusive scaling of this system converges in distribution to the Brownian web.
\end{abstract}


\section{Introduction}
Random directed spanning trees and their scaling limits have attracted lots of attention in recent years. Interest in these graphs arose because of their connection to drainage networks in geology (see Rodriguez-Iturbe and Rinaldo \cite{Rodriguez}). For graphs with the vertex set on the integer lattice, the Howard's model of drainage networks was studied by Gangopadhyay \textit{et al.} \cite{Gangopadhyay},  a  more general model was studied by Coletti and Valle \cite{Coletti2014} and a ``torch-light'' model was studied by Athreya \textit{et al.} \cite{Athreya}.  Ferrari \textit{et al.} \cite{Ferrari}
and Coupier and Tran \cite{Coupier0} studied models with vertices being points of a Poisson point process on the continuum space. These models besides being of intrinsic interest are also related to the Brownian web.

Athreya \textit{et al.} \cite{Athreya} considered a drainage network in which a source of water is connected to the nearest (in terms of the minimality of last coordinate) source lying in a $ 45^{\circ} $ light cone generating from it. We first describe the model studied by Athreya  \textit{et al.}  \cite{Athreya}. Consider the integer lattice $ \Z^{d} $ with $ d\geq 2 $. For $ \mathbf{u}=(\bu(1),\ldots,\bu(d))\in \mathbb{Z}^{d} $ and $ k,h\geq 1 $, let
\begin{itemize}
\item[(i)] $ m_{k}(\bu)=(\bu(1),\ldots,\bu(d-1),\bu(d)+k) $;
\item[(ii)] $H(\mathbf{u},k)=\big\{\mathbf{v}\in \mathbb{Z}^{d}:\bv(d)=\bu(d)+k\text{~and~}\Vert \bv-m_{k}(\bu)\Vert_{1}\leq k\big\}$ and as a convention $ H(\mathbf{u},0)=\emptyset $ (see Figure \ref{1});
\item[(iii)] $ V (\mathbf{u},h)=\bigcup_{k=1}^{h}H(\mathbf{u},k) $ and as a convention $ V (\mathbf{u},0)=\emptyset $ (see Figure \ref{1});
\item[(iv)] $ V(\bu)=\bigcup_{h=1}^{\infty}V(\mathbf{u},h) $.
\end{itemize}
Here and henceforth $ \Vert .\Vert_1 $ denotes the $ \mathcal{L}_1 $ norm on $ \R^{d} $.

\vspace*{3mm}
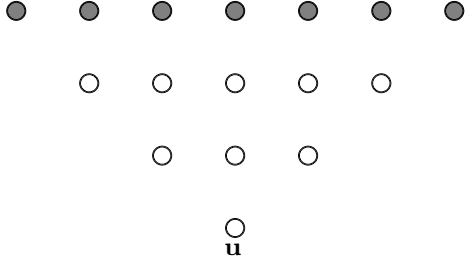
\begin{figure}[htb]
\centering
\scalebox{1.2}  
{
\begin{pspicture}(0,-6.1710095)(5.075962,-3.153029)
\definecolor{colour0}{rgb}{0.5019608,0.5019608,0.5019608}
\psdots[linecolor=black, dotsize=0.22](1.737981,-4.89101)
\psdots[linecolor=black, dotsize=0.22](2.537981,-4.89101)
\psdots[linecolor=black, dotsize=0.22](3.337981,-4.89101)
\psdots[linecolor=black, dotsize=0.22](0.93798095,-4.0910096)
\psdots[linecolor=black, dotsize=0.22](1.737981,-4.0910096)
\psdots[linecolor=black, dotsize=0.22](2.537981,-4.0910096)
\psdots[linecolor=black, dotsize=0.22](3.337981,-4.0910096)
\psdots[linecolor=black, dotsize=0.22](4.137981,-4.0910096)
\psdots[linecolor=black, dotsize=0.22](4.937981,-3.2910097)
\psdots[linecolor=black, dotsize=0.22](4.137981,-3.2910097)
\psdots[linecolor=black, dotsize=0.22](3.337981,-3.2910097)
\psdots[linecolor=black, dotsize=0.22](2.537981,-3.2910097)
\psdots[linecolor=black, dotsize=0.22](1.737981,-3.2910097)
\psdots[linecolor=black, dotsize=0.22](0.93798095,-3.2910097)
\psdots[linecolor=black, dotsize=0.22](0.13798095,-3.2910097)
\psdots[linecolor=black, dotsize=0.22](2.537981,-5.6910095)
\psdots[linecolor=black, dotstyle=o, dotsize=0.2, fillcolor=white](3.337981,-4.89101)
\psdots[linecolor=black, dotstyle=o, dotsize=0.2, fillcolor=white](2.537981,-4.89101)
\psdots[linecolor=black, dotstyle=o, dotsize=0.2, fillcolor=white](1.737981,-4.89101)
\psdots[linecolor=black, dotstyle=o, dotsize=0.2, fillcolor=white](0.93798095,-4.0910096)
\psdots[linecolor=black, dotstyle=o, dotsize=0.2, fillcolor=white](1.737981,-4.0910096)
\psdots[linecolor=black, dotstyle=o, dotsize=0.2, fillcolor=white](2.537981,-4.0910096)
\psdots[linecolor=black, dotstyle=o, dotsize=0.2, fillcolor=white](3.337981,-4.0910096)
\psdots[linecolor=black, dotstyle=o, dotsize=0.2, fillcolor=white](4.137981,-4.0910096)
\psdots[linecolor=black, dotstyle=o, dotsize=0.2, fillcolor=colour0](4.937981,-3.2910097)
\psdots[linecolor=black, dotstyle=o, dotsize=0.2, fillcolor=colour0](4.137981,-3.2910097)
\psdots[linecolor=black, dotstyle=o, dotsize=0.2, fillcolor=colour0](3.337981,-3.2910097)
\psdots[linecolor=black, dotstyle=o, dotsize=0.2, fillcolor=colour0](2.537981,-3.2910097)
\psdots[linecolor=black, dotstyle=o, dotsize=0.2, fillcolor=colour0](1.737981,-3.2910097)
\psdots[linecolor=black, dotstyle=o, dotsize=0.2, fillcolor=colour0](0.93798095,-3.2910097)
\psdots[linecolor=black, dotstyle=o, dotsize=0.2, fillcolor=colour0](0.13798095,-3.2910097)
\psdots[linecolor=black, dotstyle=o, dotsize=0.2, fillcolor=white](2.537981,-5.6910095)
\rput[bl](2.417981,-5.99){{\scriptsize $ \bu $}}
\end{pspicture}
}
\caption{\label{1}\textit{The shaded points at the top constitute $ H(\bu , 3) $ and the
 fifteen vertices above $ \bu\in\Zz $ form the region $ V (\bu ,3) $.}}
\end{figure}

Now, let $ \{U_{\bv}:\bv\in\Z^{d}\} $ and $ \{U_{\bv,\bw}:\bv,\bw\in\Z^{d}, \bw\in V(\bv)\} $ be two independent collections of i.i.d. uniform $ (0,1) $ random variables.  Fix $ 0<p<1 $, and consider $ \mathscr{V}:=\{\bw\in\Z^{d}:U_{\bw}<p\} $ as the set of all open vertices. Athreya \textit{et al.} \cite{Athreya}  have studied  a random graph with vertex set $ \mathscr{V} $ and the edge set $ \{\langle\bu ,\tilde{\bu}\rangle:\bu\in\mathscr{V}\} $, where $ \langle\bu ,\tilde{\bu}\rangle $ denotes the edge represented by a straight line joining $ \bu $ and $\tilde{\bu} $ and 
 for each $ \bu\in\Z^{d} $,  $  \tilde{\bu}\in\mathscr{V} $ is such that for some $ l\geq 1 $, 
\begin{itemize}
\item[(i)] $  \tilde{\bu}\in V(\bu,l) $;
\item[(ii)] $ V(\bu,l-1)\cap \mathscr{V}=\emptyset $;
\item[(iii)] $ U_{\bu, \tilde{\bu}}\leq U_{\bu,\bw}~\text{for all } \bw\in H(\bu,l)\cap\mathscr{V} $.
\end{itemize}

 We consider a variation of the above model. The modification is that instead of the edges being obtained via the uniform random variables $ U_{\bv ,\bw} $, we impose a structure on $ \Z^{d} $ obtained from a collection of i.i.d. uniform random variables which determine both the vertices and the edges of the random graph simultaneously.
In particular, for $\bu\in\Z^{d}$ let $\tilde{\bu}\in\mathscr{V}$ be such that for some  $l\geq 1$,
\begin{align}
\label{new1}
\text{(i)} & \quad  \tilde{\bu}\in V(\bu,l), \nonumber\\
\text{(ii)} & \quad V(\bu,l-1)\cap \mathscr{V}=\emptyset ,  \nonumber\\
\text{(iii)} & \quad U_{ \tilde{\bu}}\leq U_{\bw}~\text{for all } \bw\in H(\bu,l) \cap\mathscr{V} .
\end{align}
Note $ \P\{U_{\bv}\neq U_{\bw}\}=1 $ for all $ \bv\neq \bw $. So $  \tilde{\bu} $ is unique almost surely  and is a function of both $ \bu $ and $ \{U_{\bw}:\bw\in V(\bu)\} $. We study the random graph $ \mathscr{G}=(\mathscr{V},\mathscr{E}) $ with $ \mathscr{E}:=\{\langle\bu ,\tilde{\bu}\rangle:\bu\in\mathscr{V}\} $, where $ \tilde{\bu} $ is obtained by the second mechanism described above. Note this latter graph $ \mathscr{G} $ differs from that studied by Athreya \textit{et al.} \cite{Athreya}  for all dimensions $ d\geq 2 $. In particular, for two-dimensions while the original graph could have edges crossing each other, the variant we study does not allow that and as such is a planar graph.  Indeed, suppose there exist two distinct edges $ \langle\bx,\by\rangle $ and $ \langle\bz,\bw\rangle $ which cross each other (note that $\bx(2) < \by(2)$ and $\bz(2) < \bw(2)$). Since the two edges cross each other, we must have  $ \{\by ,\bw\}\subseteq V(\bx)\cap V(\bz)  $.
Thus, $\by(2) = \bw(2)$ and
$\tilde{\bx}=\tilde{\bz}= \left\{
\begin{array}{ll}
\by & \text{if }U_{\by}\leq U_{\bw} \\
\bw & \text{if }U_{\bw}<U_{\by}
\end{array} \right.
$, where $ \tilde{\bx} $ and $ \tilde{\bz} $ are as \eqref{new1}.

Akin to the result of Gangopadhyay \textit{et al.} \cite{Gangopadhyay}, and that of Athreya \textit{et al.} \cite{Athreya}, we have
\begin{theorem}\label{tree}
The random graph $ \mathscr{G} $ on $ \Z^{d} $,
\begin{description}
\item[(i)] for $ d=2,3 $,  is connected  and consists of a single tree almost surely.
\item[(ii)] for $ d\geq 4 $, is disconnected  and consists of infinitely many disjoint  trees almost surely.
\item[(iii)] for $ d\geq 2 $, contains no bi-infinite path almost surely.
\end{description}
\end{theorem}

Our next result studies the scaling limit of the two-dimensional graph $ \mathscr{G} $.
For $ \bu\in\Z^{ 2} $ and $ k\geq 1 $, define  $ \mathtt{u}_{0}:=\bu $ and $ \mathtt{u}_{1}:=\tilde{\bu} $ where $ \tilde{\bu} $ is as \eqref{new1}. Having obtained $ \mathtt{u}_{k} $, let $ \mathtt{u}_{k+1}:=\tilde{\mathtt{u}}_{k} $.
Starting from a vertex $ \bu\in \Zz  $, the piecewise linear path
$ \pi_{\bu}:[\bu(2),\infty )\rightarrow \R $ is defined as 
 $ \pi_{\bu}(\mathtt{u}_{k}(2)):=\mathtt{u}_{k}(1) $, for every $k \geq 0$, and $ \pi_{\bu} $ is linear in the interval $ [\mathtt{u}_{k}(2), \mathtt{u}_{k+1}(2)] $. The set of all paths comprising the graph $ \mathscr{G} $ is denoted by $ \mathcal{X}:=\{\pi_{\bu}:\bu\in\mathscr{V}\} $. We define diffusive paths as  follows:
\begin{definition} 
For $ n\geq 1 $ and normalizing constants $ \sigma ,\gamma >0 $ and for a path  $ \pi $  starting from $  \varsigma_{\pi} $, the scaled path $ \pi^{(n)}(\sigma ,\gamma) $ is defined by
\[ \pi^{(n)}(\sigma ,\gamma):[ \varsigma_{\pi}/(n^{2}\gamma),\infty )\rightarrow \R\] 
such that 
$ \pi^{(n)}(\sigma ,\gamma)(t)=\pi(n^{2}\gamma t)/(n\sigma ) $.
\end{definition}

The collection of scaled paths is $ \mathcal{X}_{n} (\sigma ,\gamma):=\{\pi_{\bu}^{(n)}(\sigma ,\gamma):\bu\in\mathscr{V}\} $. 
In Theorem \ref{BW} we show that, for $d=2$, the scaling limit of the graph $ \mathscr{G} $  is the Brownian web.

The Brownian web is a collection of coalescing Brownian motions starting from each space-time point in $ \R^{2} $. 
It was  introduced by Arratia \cite{Arratia1,Arratia2} to study the asymptotic behaviour of the one-dimensional voter model.  T\'{o}th and Werner \cite{Toth} described a version of the Brownian web in their study of the true self repelling motion. For a comprehensive review see Schertzer {\it et al.}\/  \cite{Sun}. The Brownian web we describe here is as in Fontes \textit{et al.} \cite{Fontes1, Fontes2}.

Let $\R^{2}_c$ denote the completion of the space time plane $\R^2$ with
respect to the metric
\begin{equation}
\label{rhodef}
 \rho((x_1,t_1),(x_2,t_2)) := |\tanh(t_1)-\tanh(t_2)|\vee \Bigl| \frac{\tanh(x_1)}{1+|t_1|}
 -\frac{\tanh(x_2)}{1+|t_2|} \Bigr|.
\end{equation}
 The topological space $\R^{2}_c$ can be identified with the
continuous image of $[-\infty,\infty]^2$ under a map that identifies the line
$[-\infty,\infty]\times\{\infty\}$ with  a point $(\ast,\infty)$, and the line
$[-\infty,\infty]\times\{-\infty\}$ with a point $(\ast,-\infty)$.
A path $\pi$ in $\R^{2}_c$ with starting time $ \varsigma_{\pi}\in [-\infty,\infty]$
is a mapping $\pi :[ \varsigma_{\pi},\infty]\rightarrow [-\infty,\infty] \cup \{ \ast \}$ such that
$\pi(\infty)= \ast$ and, when $ \varsigma_{\pi} = -\infty$, $\pi(-\infty)= \ast$;
Also $t \mapsto (\pi(t),t)$ is a continuous
map from $[ \varsigma_{\pi},\infty]$ to $(\R^{2}_c,\rho)$.
We then define $\Pi$ to be the space of all paths in $\R^{2}_c$ with all possible starting times in $[-\infty,\infty]$.
The following metric 
\begin{equation*}
d_{\Pi} (\pi_1,\pi_2) := |\tanh(\varsigma_{\pi_1})-\tanh(\varsigma_{\pi_2})|\vee\sup_{t\geq
\varsigma_{\pi_1}\wedge
\varsigma_{\pi_2}} \Bigl|\frac{\tanh(\pi_1(t\vee\varsigma_{\pi_1}))}{1+|t|}-\frac{
\tanh(\pi_2(t\vee\varsigma_{\pi_2}))}{1+|t|}\Bigr|
\end{equation*}
 (for $\pi_1,\pi_2\in \Pi$), makes $\Pi$ a complete, separable metric space. 
Convergence in this metric can be described as locally uniform convergence of paths 
as well as convergence of starting points.

Let ${\mathcal H}$ be the space of compact subsets of $(\Pi, d_{\Pi})$ equipped with
the Hausdorff metric $d_{{\mathcal H}}$ given by
\begin{equation*}
d_{{\mathcal H}}(K_1,K_2) := \sup_{\pi_1 \in K_1} \inf_{\pi_2 \in
K_2} d_{\Pi} (\pi_1,\pi_2)\vee
\sup_{\pi_2 \in K_2} \inf_{\pi_1 \in K_1}  d_{\Pi} (\pi_1,\pi_2);
\end{equation*}
 $({\mathcal H},d_{{\mathcal H}})$ is a complete, separable metric space. Let
$ \mathcal{B}_{{\mathcal H}}$ be the Borel  $\sigma$-algebra on the metric space $({\mathcal H},d_{{\mathcal H}})$.
The Brownian web ${\mathcal W}$ is characterized as:
\begin{theorem}
\label{theorem:Bwebcharacterisation}
\textup{\textbf{(Fontes \textit{et al.} \cite{Fontes2})}} There exists an $({\mathcal H}, {\mathcal B}_{{\mathcal H}})$ valued random variable
${\mathcal W}$ whose distribution  is uniquely determined by
the following properties:
\begin{itemize}
\item[$(a)$] for each deterministic point $\bz\in \R^2$,
there is a unique path $ \pi^{\bz}\in {\mathcal W}$  almost surely;
\item[$(b)$] for a finite set of deterministic points $ \bz_1,\dotsc, \bz_k \in \R^2$,
the collection $(\pi^{\bz_1},\dotsc,\pi^{\bz_k})$ is distributed as coalescing Brownian motions  starting from $ \bz_1,\dotsc, \bz_k $;
\item[$(c)$] for any countable deterministic dense set ${\mathcal D} \subseteq   \R^2$, 
${\mathcal W}$ is the closure of $\{ \pi^{\bz}: \bz\in {\mathcal D} \}$ in $(\Pi,  d_{\Pi})$  almost surely.
\end{itemize}
\end{theorem}

Theorem \ref{theorem:Bwebcharacterisation} shows that the collection  $ {\mathcal W} $ is 
almost surely determined by countably many coalescing Brownian motions.  The scaling limit of the two-dimensional graph $ \mathscr{G} $ is given by the following theorem:
 \begin{theorem}\label{BW}
There exist $ \sigma:=\sigma(p) $ and $ \gamma:=\gamma (p) $ such that,  $ \bar{\mathcal{X}}_{n}(\sigma ,\gamma) $ converges in distribution to the standard Brownian web $\mathcal{W} $   as $ n\rightarrow \infty $, where  $ \bar{\mathcal{X}}_{n}(\sigma ,\gamma) $ is the closure of $ \mathcal{X} _{n}(\sigma , \gamma) $ in $ (\Pi ,d_{\Pi}) $.
\end{theorem}

The rest of this paper is organized as follows. In Section \ref{s2} we present a construction of the graph.   This construction, which is different from that of Athreya \textit{et al.} \cite{Athreya}, brings out the Markov property inherent in the process through regeneration times and is the basis of the proofs of Theorems \ref{tree} and  \ref{BW} in Sections \ref{s3} and \ref{s4} respectively.

\section{A construction of the graph}\label{s2}
 We proceed to the construction of the graph which is based on the construction given in Roy \textit{et al.} \cite{Roy2016}.  We state the construction of the graph starting from two vertices, the general case for $k$ vertices follows in a similar fashion.

Fix two vertices $ \bu ,\bv\in\Z^{d} $ (not necessarily open) with $ \bu( d)=\bv( d) $. Let $h_0(\bu) := \bu$, $h_0(\bv) := \bv$, $\Delta_0:= \emptyset$,  $r_0  = s_0= \bu(d)$ and $\Psi_0: \Delta_0 \to [0,1]$ the empty function. Take
\begin{enumerate}
\item $h_1(\bu) := \tilde{\bu}$,
$h_1(\bv) := \tilde{\bv}$,  where $ \tilde{\bu} $ and $ \tilde{\bv} $ are as \eqref{new1},
\item $r_1 := \min\{h_1(\bu)(d),  h_1(\bv)(d)\}$, $s_1 := \max\{h_1(\bu)(d),  h_1(\bv)(d)\}$, \item $\Delta_1 := \begin{cases} V(h_0(\bu), s_1) \setminus V(h_0(\bu),r_1) & \text{if } h_1(\bu)(d) \geq h_1(\bv)(d)\\
V(h_0(\bv), s_1) \setminus V(h_0(\bv),r_1) & \text{if } h_1(\bv)( d) > h_1(\bu)( d)\end{cases}
$ ,
\item $\Psi_1:\Delta_1 \to [0,1]$ given by $\Psi_1(\bw) = U_{\bw}$ for $\bw \in \Delta_1$, where $\Psi_1$ is taken to be the empty function if $\Delta_1 = \emptyset$.
\end{enumerate}

Having defined $h_n(\bu)$, $h_n(\bv)$, $r_n$, $s_n$, $\Delta_n$ and $\Psi_n$, we obtain
\begin{enumerate}
\item $h_{n+1}(\bu) := \begin{cases}  \tilde{ h}_n(\bu) & \text{if } h_n(\bu)(d) = r_n\\
h_n(\bu) & \text{if } h_n(\bu)(d) > r_n\end{cases}$, \\$h_{n+1}(\bv) := \begin{cases}\tilde{ h}_n(\bv) & \text{if } h_n(\bv)(d) = r_n\\
h_n(\bv) & \text{if } h_n(\bv)(d) > r_n\end{cases}$ ,\\where $ \tilde{ h}_n(\bu) $ and $ \tilde{ h}_n(\bv) $ are as \eqref{new1},
\item $r_{n+1} := \min\{h_{n+1}(\bu)(d),  h_{n+1}(\bv)(d)\}$, $s_{n+1} := \max\{h_{n+1}(\bu)(d),  h_{n+1}(\bv)(d)\}$, \item $\Delta_{n+1} := \begin{cases} V(h_n(\bu), s_{n+1}) \setminus V(h_n(\bu),r_{n+1}) & \text{if } h_{n+1}(\bu)(d) \geq h_{n+1}(\bv)(d)\\
V(h_n(\bv), s_{n+1}) \setminus V(h_n(\bv),r_{n+1}) & \text{if } h_{n+1}(\bv)(d) > h_{n+1}(\bu)(d)\end{cases}
$,
\item $\Psi_{n+1}:\Delta_{n+1} \to [0,1]$ given by $\Psi_{n+1}(\bw) = U_{\bw}$ for $\bw \in \Delta_{n+1}$, where $\Psi_{n+1}$ is taken to be the empty function if $\Delta_{n+1} = \emptyset$.
\end{enumerate}

We elaborate on the formal construction above with the help of Figure \ref{2}. Starting from $ \bu$ and $\bv $, which are at the same level, we get the vertices $ h_{1}(\bu) $ and $ h_{1}(\bv) $ to obtain the edges $ \langle \bu, h_{1}(\bu) \rangle$ and $ \langle \bv, h_{1}(\bv) \rangle$. At this stage we observe that $h_{1}(\bu)$ is below $h_{1}(\bv)$, and we now explore the vertices above $h_{1}(\bu)$ to obtain $h_{2}(\bu)$. During this exploration process we have information about the shaded region $\Delta_1$, which we call the `history region' at this stage. Having obtained the edge  $ \langle h_{1}(\bu), h_{2}(\bu) \rangle$, we rename $h_{1}(\bv)$ as $h_{2}(\bv)$. Next we observe that $h_{2}(\bv)$ is below $h_{2}(\bu)$ and so we start exploring from $h_{2}(\bv)$ with $\Delta_2$ as the history region until we obtain $h_{3}(\bv)$, when we rename $h_{2}(\bu)$ as $h_{3}(\bu)$. In the figure 
$h_{3}(\bu)$  and $h_{3}(\bv)$ are at the same level, so that the history region $\Delta_3$ is empty. We start afresh from these two vertices and proceed as above. As will be seen later (see the paragraph leading to \eqref{r-regen}) when the history region is empty, we have a regeneration step. At stage $n$ of this construction, the information of the uniform random variables in the history region is specified by the function $ \Psi_{n} $.

\vspace*{3mm}
\begin{figure}[htb]
\centering
\scalebox{1}  
{
\begin{pspicture}(0,-5.11)(12.557383,0.03)
\definecolor{colour0}{rgb}{0.5019608,0.5019608,0.5019608}
\definecolor{colour1}{rgb}{0.8,0.8,0.8}
\psline[linecolor=black, linewidth=0.014](9.6,-2.63)(7.6,-4.63)(5.6,-2.63)
\psline[linecolor=black, linewidth=0.014](3.6,-4.63)(5.2,-3.03)(2.0,-3.03)(3.6,-4.63)
\psline[linecolor=colour0, linewidth=0.024, fillstyle=solid,fillcolor=colour1](5.8,-2.83)(9.4,-2.83)(10.8,-1.43)(4.4,-1.43)(5.8,-2.83)
\psline[linecolor=black, linewidth=0.044](7.6,-4.63)(8.8,-1.43)
\psline[linecolor=black, linewidth=0.044](3.6,-4.63)(3.6,-3.03)
\psline[linecolor=black, linewidth=0.014](5.6,-1.03)(3.6,-3.03)(1.6,-1.03)
\psline[linecolor=colour0, linewidth=0.024, fillstyle=solid,fillcolor=colour1](1.8,-1.23)(5.4,-1.23)(6.0,-0.63)(1.2,-0.63)(1.8,-1.23)
\psline[linecolor=black, linewidth=0.044](3.6,-3.03)(2.4,-0.63)
\psdots[linecolor=black, dotsize=0.12](7.6,-4.63)
\psdots[linecolor=black, dotsize=0.12](3.6,-4.63)
\psdots[linecolor=black, dotsize=0.12](3.6,-3.03)
\psdots[linecolor=black, dotsize=0.12](2.4,-0.63)
\psdots[linecolor=black, dotsize=0.12](8.8,-1.43)

\psline[linecolor=black, linewidth=0.014](8.8,-1.43)(9.6,-0.63)(8.0,-0.63)(8.8,-1.43)
\psline[linecolor=black, linewidth=0.044](8.8,-1.43)(8.4,-0.63)
\psdots[linecolor=black, dotsize=0.12](8.4,-0.63)
\psline[linecolor=black, linewidth=0.03, linestyle=dotted, dotsep=0.10583334cm](0.8,-3.03)(11.2,-3.03)
\psline[linecolor=black, linewidth=0.03, linestyle=dotted, dotsep=0.10583334cm](0.8,-1.43)(11.2,-1.43)
\psline[linecolor=black, linewidth=0.03, linestyle=dotted, dotsep=0.10583334cm](0.8,-0.63)(11.2,-0.63)
\rput[bl](7.3,-2.26){{\footnotesize $ \Delta_{1} $}}
\rput[bl](3.3,-1.08){{\footnotesize $ \Delta_2 $}}
\rput[bl](0.3,-3.09){$ r_{1} $}
\rput[bl](0.3,-1.49){$ r_{2} $}
\rput[bl](0.3,-0.69){$ r_{3} $}
\rput[bl](3.45,-4.95){$ \bu $}
\rput[bl](7.45,-4.95){$ \bv $}
\rput[bl](1.3,-0.45){{\footnotesize $ h_{2}(\bu)=h_{3}(\bu) $}}
\rput[bl](8.0,-0.45){{\footnotesize $ h_{3}(\bv) $}}
\psline[linecolor=black, linewidth=0.03, arrowsize=0.05291667cm 2.0,arrowlength=1.4,arrowinset=0.0]{->}(8.95,-1.35)(9.5,-1.1)
\rput[bl](9.5,-1.2){{\footnotesize $ h_{1}(\bv)=h_{2}(\bv) $}}
\rput[bl](2.0,-2.75){{\footnotesize $ h_{1}(\bu) $}}
\psline[linecolor=black, linewidth=0.03, arrowsize=0.05291667cm 2.0,arrowlength=1.4,arrowinset=0.0]{->}(3.45,-2.95)(2.9,-2.7)
\rput[bl](6.0,-0.23){{\footnotesize $ \Delta_3=\emptyset $}}
\end{pspicture}
}
\caption{\label{2}\textit{The first three steps of a process started with $ \bu ,\bv\in\Zz $.}}
\end{figure}

Having constructed $\mathcal{Z}_{n} :=(h_{n}(\bu),h_{n}(\bv),\Delta_{n},\Psi_{n})$ for $ n\geq 0 $, we do not have any information about the region $\{\by \in \Z^d: \by(d)  > r_n\}$ except for the trapezoidal region $\Delta_n$. About the trapezoidal region $\Delta_n$, we know that there is at least one vertex $\bw \in \Delta_n \cap \{\by \in \Z^d: \by(d)  = s_n\}$ with $U_\bw <p$ and for all vertices $\bz \in \Delta_n \cap \{\by \in \Z^{d}: \by(d)  < s_n\}$ we have $U_\bz  \geq p$. Now let ${\cal D}_\bu := \{V(\bu,s) \setminus V(\bu,r) : r \leq s,\; r,s \in \Z\}$, and ${\cal D} := \cup_{\bu \in \Z^{d}} {\cal D}_\bu$. Consider the space
\begin{align*}
\mathfrak{G}:=&\Big\{(\bu,\bv,\Delta ,\Psi):\bu,\bv\in\Z^{d},
 \Delta\in {\cal D},\text{ and }
\Psi :\Delta\rightarrow [0,1]\text{ is a mapping with }\\& \Psi(\bw)\geq p \text{  for  all  }\bw\in\Delta \text{ with } \bw (d) < \max\{\bz(d) : \bz \in \Delta\}, \;  \Psi(\bw) < p \text{ for} \\&
\text{at least one vertex } \bw  \in  \Delta
\text{ and }\Psi \text{ is the empty function when }\Delta=\emptyset
\Big\}.
\end{align*}
We then have the following theorem:
\begin{theorem}\label{beingMP}
The process $ \mathcal{Z} := \{\mathcal{Z}_{n} : n \geq 0\}$ is a Markov process on the state space $ \mathfrak{G} $.
\end{theorem}
\begin{proof}
Let  $ Y:=\{V_{\bw}:\bw\in\Z^{d},\bw(d)>0\} $ be a set of i.i.d. uniform (0,1) random variables, independent of the collection $ \{U_{\bw}:\bw\in\Z^{d}\} $ used to build the model. Since at the end of the $ n $-th step all vertices in $ \Z^{d}\setminus (\Delta_{n}\cup H(r_{n})) $ are unexplored,  we can assign another i.i.d. uniform (0,1) random variables instead of them, and hence we can find a function $ f:[0,1]^{\Z^{d}\setminus H(0)}\times \mathfrak{G}\rightarrow [0,1]^{\Z^{d}} $ such that $ \{U_{\bw}:\bw\in\Z^{d},\bw(d)>r_{n}\} ~\big |~ \mathcal{Z}_{n}\stackrel{d}{=} f(Y,\mathcal{Z}_{n}) $. Thus, $ \mathcal{Z}_{n+1} ~\big |~ \{\mathcal{Z}_{j}:j=0,1,\ldots ,n\}\stackrel{d}{=} g_{n}(\mathcal{Z}_{n},f(Y,\mathcal{Z}_{n})) $ for some function $ g_{n} : \mathfrak{G}\times [0,1]^{\Z^{d}\setminus H(r_{n})} \to \mathfrak{G} $, that means $ \mathcal{Z} $ is a Markov process \cite{Levin}. 
\end{proof}

When $\Delta_n = \emptyset$, then $h_n(\bu)(d) = h_n(\bv)(d)$ and the entire half-space  above these two vertices are unexplored. We will exploit this renewal property repeatedly, and so we define
 $
\tau_{0}:=0 $ and
\begin{equation}
\label{r-regen}
\tau_{l}=\tau_{l}(\bu ,\bv):=\inf\{n>\tau_{l-1}:\Delta_{n}=\emptyset\} =
\inf\{n>\tau_{l-1}:r_n = s_n\} \text{ for $l \geq 1$}. \end{equation}

The total number of steps between the $ (l-1) $-th and $ l $-th renewal of the process  is denoted by
\[\sigma_{l}=\sigma_{l}(\bu ,\bv):=\tau_{l}-\tau_{l-1}, \]
and the time required for the $ l $-th renewal  is denoted by
\begin{equation}
\label{rrenewal}
T_{l}=T_{l}(\bu ,\bv):=h_{\tau_{l}}(\bu)(d)-\bu(d)=h_{\tau_{l}}(\bv)(d)-\bv(d).
\end{equation}

From the renewal nature of the process, we have that $ \{\sigma_{l}:l \geq 1\} $ and $ \{T_{l+1}-T_{l}:l\geq 0 \} $ are collections of independent random variables.
The next proposition shows that $ \sigma_{l} $ and $ T_{l+1}-T_{l} $ have exponentially decaying tail probabilities. The following two lemmas will be used in the proof of the proposition. 

\begin{lemma}\textup{\textbf{(Lemma 2.6 of \cite{Roy2016})}}\label{Lemma 2.6 of RSS} 
Let $ \{\theta_{n}:n\geq 1\} $ be a sequence of i.i.d. positive integer valued random variables with $ \P\{\theta_{1}=1\}>0 $ and $ \P\{\theta_{1}\geq m\}\leq C_{1,1}\exp\{-C_{1,2}m\}$ for all $ m\geq 1 $, where $ C_{1,1},C_{1,2} $ are some positive constants. For a sequence of random variables $ \{M_{l}:l\geq 0\} $  defined by $ M_{0}:=0 $ and $ M_{l+1}:=\max\{M_{l},\theta_{l+1}\}-1 $, set $ \tau^{M}:=\inf\{l\geq 1 :M_{l}=0\}  $. Then, we have 
\[\P\{\tau^{M}\geq m\}\leq C_{1,3}\exp\{-C_{1,4}m\}\]
for any $ m\geq 1 $, where $ C_{1,3} $ and $ C_{1,4} $ are some positive constants.
\end{lemma}

\begin{lemma}\textup{\textbf{(Lemma 2.7 of \cite{Roy2016})}} \label{Lemma 2.7 of RSS}
Let $ \{\theta_{n}:n\geq 1\} $ be a sequence of i.i.d. random variables and $ N $ be a random variable taking values in $ \N $. If $ \E[\mathrm{e}^{\alpha_1 N}]<\infty $ and  $ \E[\mathrm{e}^{\alpha_2 \theta_{1}}]<\infty $ for some $ \alpha_1 >0 $ and $ \alpha_2 >0 $,  then there exists some $ \gamma >0 $ such that $ \E[\mathrm{e}^{\gamma \sum_{n=1}^{N}\theta_{n}}]<\infty $.
\end{lemma}

\begin{proposition}\label{ppppp}
For all $ m \geq 1$, and  some positive constants  $ C_{1}, C_{2} , C_3$ and $C_4$, we have
$\P\{\sigma_{l}\geq m\}\leq C_{1}\exp\{-C_{2}m\}$ and $\P\{T_{l+1} - T_{l}\geq m\}\leq C_{3}\exp\{-C_{4}m\}$.
\end{proposition}
\begin{proof} 
We prove it for $ l=1 $. The general case being similar. For $ n\geq 0 $ and $ \bx\in\{\bu ,\bv\} $ define
\begin{align*}
m^{-}_n (\bx)&:= \big\{ \big(h_{n}(\bx)(1)-j,\ldots ,h_{n}(\bx)(d-1)-j,h_{n}(\bx)(d)+j\big):j> 0\big\},\\
m^{+}_n(\bx)&:=\big\{ \big(h_{n}(\bx)(1)+j ,\ldots ,h_{n}(\bx)(d-1)+j,h_{n}(\bx)(d)+j\big):j>0\big\}.
\end{align*}

First let $n$ be such that $ \Delta_{n}=\emptyset $, i.e. $ h_{n}(\bu)(d)=h_{n}(\bv)(d) $. In this case, we must have either $ m_{n}^{-}(\bu)\cap V(h_{n}(\bv))=\emptyset $ or $ m_{n}^{+}(\bu)\cap V(h_{n}(\bv))=\emptyset $. Set
\begin{align*}
m_{n}(\bu)= \left\{
\begin{array}{ll}
m_{n}^{-}(\bu) & \text{if }m_{n}^{-}(\bu)\cap V(h_{n}(\bv))=\emptyset , \\
m_{n}^{+}(\bu) & \text{if }m_{n}^{+}(\bu)\cap V(h_{n}(\bv))=\emptyset ,
\end{array} \right.
\end{align*}
and define  $ m_{n}(\bv) $ similarly. For $ n\geq 0 $ and $ \bx\in\{\bu ,\bv\} $, taking $\bx_{n,1}, \bx_{n,2}, \ldots$ to be the vertices of $\Z^{d}$ in $m_n (\bx)$ ordered in terms of their proximity to $h_n(\bx)$, and
$$
J_{n+1}(\bx):=\inf\big\{j\geq 1 : \bx_{n,j} \in\mathscr{V}\big\}, \text{ a  geometric $(p)$ random variable},
$$
we  have $ h_{n+1}(\bx)(d)-h_{n}(\bx)(d)\leq J_{n+1}(\bx) $. Set 
$$ J_{n+1}:=\max\{J_{n+1}(\bu),J_{n+1}(\bv)\}. $$

Next suppose $n$ is such that $ \Delta_{n}\neq\emptyset $. Without loss of generality suppose $h_n(\bu)(d) < h_n(\bv)(d)$. Thus  by our construction, $h_{n+1}(\bv) = h_n(\bv)$ and $h_{n+1}(\bu) (d) > h_n(\bu) (d) $.   It must be the case that either $m^{-}_n (\bu) \cap  V(h_{ \mathtt{m}_{0}}(\bv)) = \emptyset$ or 
$m^{+}_n(\bu) \cap V(h_{\mathtt{m}_{0}}(\bv)) = \emptyset$, where $ \mathtt{m}_{0}:=\max\{m\geq 0:h_{m}(\bv)(d)\leq h_{n}(\bu)(d)\} $ (see Figure \ref{prop2.2}).
Here let $m_n (\bu)$  denote $m^{-}_n (\bu)$ if $m^{-}_n (\bu) \cap  V(h_{\mathtt{m}_{0}}(\bv)) = \emptyset$, otherwise it denotes $m^{+}_n (\bu)$. Also, let $J_{n+1}(\bu)$  be defined as before. We have
$$
h_{n+1}(\bv)(d)-h_{n}(\bv)(d)=0 \text{ and }
h_{n+1}(\bu)(d)-h_{n}(\bu)(d)\leq J_{n+1}(\bu).$$
Then, consider an independent collection of i.i.d. random variables $ \{J_{n}^{\prime}:n\geq 1\} $ so that $ J_{1}^{\prime} $ has  a geometric $(p)$ distribution, and define  $$J_{n+1}:=\max\{J_{n+1}(\bu),J_{n+1}^{\prime}\}. $$

Now we observe that, in each of the two cases above, for each $ n\geq 0 $, 
\[ \max\big\{h_{n+1}(\bu)(d)-h_{n}(\bu)(d),h_{n+1}(\bv)(d)-h_{n}(\bv)(d)\big\}\leq J_{n+1}, \]
 where  $ \{J_{n}:n\geq 1\} $ is a sequence of i.i.d. positive integer valued random variables with exponentially decaying tail probabilities.  Taking auxiliary random variables $ M_{0}:=0 $, $ M_{n+1}:=\max\{M_{n},J_{n+1}\}-1 $ for all $ n\geq 0 $, and $ \tau^{M}:=\inf\{n\geq 1 : M_{n}=0\} $, from Lemma \ref{Lemma 2.6 of RSS} we have
\begin{align}\label{p2}
\P\{\tau^{M}\geq m\}\leq C_{1}\exp\{-C_{2}m\} \text{ for all } m \geq  1  \text{ and positive constants
} C_{1} \text{ and }  C_{2}.
\end{align}

\vspace*{3mm}
\begin{figure}[htb]
\centering
\scalebox{1}  
{
\begin{pspicture}(0,-4.435072)(12.409856,2.0547838)
\definecolor{colour0}{rgb}{0.5019608,0.5019608,0.5019608}
\definecolor{colour1}{rgb}{0.7019608,0.7019608,0.7019608}
\psline[linecolor=black, linewidth=0.024, fillstyle=solid,fillcolor=colour1](11.6,1.244928)(6.8,-3.555072)(2.0,1.244928)
\psline[linecolor=colour0, linewidth=0.038](4.8,-2.355072)(1.2,1.244928)
\psdots[linecolor=black, dotsize=0.2](6.8,-3.555072)
\psdots[linecolor=black, dotsize=0.2](7.6,0.444928)
\psdots[linecolor=black, dotsize=0.2](5.6,-3.555072)
\psdots[linecolor=black, dotsize=0.2](4.8,-2.355072)
\psline[linecolor=black, linewidth=0.036](6.8,-3.555072)(7.6,0.444928)
\psline[linecolor=black, linewidth=0.036](5.6,-3.555072)(4.8,-2.355072)
\psline[linecolor=black, linewidth=0.02, linestyle=dashed, dash=0.20638889cm 0.10583334cm](5.6,-3.555072)(6.8,-2.355072)(4.4,-2.355072)(5.6,-3.555072)
\rput[bl](6.4,-3.955072){{\footnotesize $ h_{n-1}(\bv) $}}
\rput[bl](6.4,0.64492797){{\footnotesize $ h_{n}(\bv)=h_{n+1}(\bv) $}}
\rput[bl](4.9,-3.955072){{\footnotesize $ h_{n-1}(\bu) $}}
\rput[bl](3.05,-2.205072){{\footnotesize $ h_{n}(\bu) $}}
\psline[linecolor=black, linewidth=0.024, arrowsize=0.05291667cm 2.0,arrowlength=1.4,arrowinset=0.0]{->}(4.6,-2.255072)(4.0,-2.055072)
\rput[bl](0.5,0.444928){{\footnotesize \textcolor{colour0}{$  m_{n}^{-}(\bu) $}}}
\rput[bl](6,1.644928){{\footnotesize $ V(h_{n-1}(\bu)) $}}
\psdots[linecolor=colour0, dotsize=0.04](1.0,1.444928)
\psdots[linecolor=colour0, dotsize=0.04](0.8,1.644928)
\psdots[linecolor=colour0, dotsize=0.04](0.6,1.844928)
\psdots[linecolor=black, dotsize=0.04](11.8,1.444928)
\psdots[linecolor=black, dotsize=0.04](12.0,1.644928)
\psdots[linecolor=black, dotsize=0.04](12.2,1.844928)
\psdots[linecolor=black, dotsize=0.04](1.8,1.444928)
\psdots[linecolor=black, dotsize=0.04](1.6,1.644928)
\psdots[linecolor=black, dotsize=0.04](1.4,1.844928)
\end{pspicture}
}
\caption{\label{prop2.2}\textit{In this realization,  $ \mathtt{m}_{0}=n-1 $ and $ m^{-}_n (\bu) \cap  V(h_{\mathtt{m}_{0}}(\bv)) = \emptyset$.}}
\end{figure}

 We will show that $ s_n - r_n \leq M_{n} $ for all $ 0\leq n\leq \sigma_{1} $, which yields $\sigma_1 \leq \tau^M$, and thus, together with \eqref{p2}, completes the proof  of the first part of the proposition.
First, $s_0 - r_0 = M_0 = 0$.    Also, $ s_{1}-r_{1}<J_{1} $, which implies that $ s_{1}-r_{1}\leq M_{1} $.
Now suppose $ s_n - r_n \leq M_{n} $ holds for  some  $ 0 < n < \sigma_{1} $.  If $ \Delta_{n+1}=\emptyset $,  then $ 0=s_{n+1} - r_{n+1} \leq M_{n+1} $. Otherwise, without loss of generality suppose $h_n(\bu)(d) < h_n(\bv)(d)$. If $h_{n+1}(\bu)(d) < h_{n+1}(\bv)(d)$, then  $ r_{n}<r_{n+1}<s_{n}=s_{n+1} $, which gives $s_{n+1} - r_{n+1} < s_n - r_n \leq M_{n} $, where the last inequality follows from the induction hypothesis. On the other hand if
$h_{n+1}(\bu)(d) > h_{n+1}(\bv)(d)$, then   $ s_{n+1}-r_{n+1}<s_{n+1}-r_{n}=h_{n+1}(\bu)(d)-h_{n}(\bu)(d)\leq J_{n+1} $. Because of the strictly inequalities, the induction step is completed.

Finally, $T_1 = r_{\tau_1} - r_0 \leq \sum_{n=1}^{\tau_1} J_n$, and so from
 Lemma \ref{Lemma 2.7 of RSS} we have that $\E[\mathrm{e}^{\gamma T_1}]<\infty $ for some $ \gamma >0 $. Using Markov inequality we have
$\P\{T_{1}\geq m\} \leq \E[\mathrm{e}^{\gamma T_{1}}]\mathrm{e}^{-\gamma m}$, which completes the proof of the proposition.
\end{proof}

 By translation invariance,
$ \{h_{\tau_{l}}(\bu)-h_{\tau_{l}}(\bv):l \geq 0\} $
 is a Markov chain  on $ \Z^{d} $, and since $ h_{\tau_{l}}(\bu)(d)=h_{\tau_{l}}(\bv)(d) $,  taking $ \bar{\bw}=(\bw(1), \ldots ,\bw(d-1)) $ for $ \bw=(\bw(1), \ldots ,\bw(d)) $,
\begin{equation}
\label{rZ} \{Z_{l}(\bu ,\bv):=\bar{h}_{\tau_{l}}(\bu)-\bar{h}_{\tau_{l}}(\bv):l  \geq 0\}
\end{equation}
is a Markov chain on $ \Z^{d-1} $ with   $ (0,\ldots ,0)\in\Z^{d-1} $ its only absorbing state.

\section{Geometry of the graph}\label{s3}
The proof of Theorem \ref{tree} is based on the ideas in Gangopadhyay \textit{et al.} \cite{Gangopadhyay}, and Athreya \textit{et al.} \cite{Athreya}, and as such we present a brief sketch of the modifications required.

\subsection{Proof of Theorem \ref{tree}(i)}
\textit{Dimension 2.} We first prove a result which we need now as well as in the next section.
\begin{theorem}
\label{rmart}
For $ \bu ,\bv\in\Zz $ with $ \bu(2)=\bv(2) $, the process $ \{Z_{l}(\bu ,\bv)=h_{\tau_{l}}(\bu)(1)-h_{\tau_{l}}(\bv)(1):l\geq 0\}  $  is a martingale with respect to the filtration $ \{\mathscr{F}_{T_{l}}:l\geq 0\} $, where $ \mathscr{F}_{t}:=\sigma(\{U_{\bw}:\bw\in\Zz,\bw(2)\leq \bu(2)+t\}) $ and $T_l$ is as in (\ref{rrenewal}).
\end{theorem}
\begin{proof}
By translation invariance, we first consider $ \bu = \mathbf{0}$. To simplify notation, throughout this proof, $h_n$ and $\bar{h}_n$ will denote the vertices $h_n(\mathbf{0})$ and $h_n(\mathbf{0})(1)$  respectively. It can be easily proved that  $ \bar{h}_{\tau_{l}} $ is $ \mathscr{F}_{T_{l}} $-measurable.
Also, the construction of our model ensures that $ \vert h_{1}(\bw)(1)-\bw(1)\vert\leq h_{1}(\bw)(2)-\bw(2) $ for every $\bw \in \Z^2$, so invoking the  random variables introduced in the  proof of Proposition \ref{ppppp}, we have
$
\vert \bar{h}_{\tau_{l}}\vert   \leq  \sum_{n=1}^{\tau_l} J_n $.
So, Lemma \ref{Lemma 2.7 of RSS} implies that $ \E[\vert\bar{h}_{\tau_{l}}\vert]<\infty $  for every $ l\geq {\text 0}$.

Finally, for any $ A\in\mathscr{F}_{T_{l}} $, and taking $ \mathcal{G}_{n}:= \mathscr{F}_{h_{n}(2)}$, we have
\begin{align*}
&\E\Big[\mathbf{1}(A)\big(\bar{h}_{\tau_{l+1}}-\bar{h}_{\tau_{l}}\big)\Big]
\\&=\E\Big[\mathbf{1}(A)\sum_{n=l}^{\infty}\sum_{m=1}^{\infty}\Big\lbrace \mathbf{1}(\{\tau_{l}=n\})\mathbf{1}(\{\tau_{l+1}=n+m\})\sum_{i=1}^{m}\big(\bar{h}_{n+i}-\bar{h}_{n+i-1}\big)\Big\rbrace\Big]
\\&=
\sum_{n=l}^{\infty}\sum_{m=1}^{\infty}\E\Big[\E\Big[\mathbf{1}(A) \mathbf{1}(\{\tau_{l}=n\})\mathbf{1}(\{\tau_{l+1}\geq n+m\})
 \big(\bar{h}_{n+m}- \bar{h}_{n+m-1}\big)
~\Big |~\mathcal{G}_{n+m-1}\Big]\Big].
\end{align*}
Noting that $ \mathbf{1}(A)\mathbf{1}(\{\tau_{l}=n\})\mathbf{1}(\{\tau_{l+1}\geq n+m\}) $ is $ \mathcal{G}_{n+m-1} $-measurable,  $ \bar{h}_{n+m}- \bar{h}_{n+m-1} $ is independent of $ \mathcal{G}_{n+m-1} $ and that the increment of  $\{\bar{h}_n:n\geq 0\}$ is symmetric, we have
$ \E\big[\bar{h}_{\tau_{l+1}}-\bar{h}_{\tau_{l}} ~\big |~\mathscr{F}_{T_{l}}\big]=0 $ almost surely. Therefore for each $ \bu ,\bv\in\Zz $ with $ \bu(2)=\bv(2) $, and for any $ l\geq 0 $,
\begin{align*}
&\E\big[Z_{l+1}(\bu ,\bv)-Z_{l}(\bu ,\bv) ~\big |~\mathscr{F}_{T_{l}}\big]\\&
=\E\big[h_{\tau_{l+1}}(\bu)(1)-h_{\tau_{l+1}}(\bv)(1)-\big(h_{\tau_{l}}(\bu)(1)-h_{\tau_{l}}(\bv)(1)\big) ~\big |~\mathscr{F}_{T_{l}}\big]\\&
=\E\big[h_{\tau_{l+1}}(\bu)(1)-h_{\tau_{l}}(\bu)(1) ~\big |~\mathscr{F}_{T_{l}}\big]-\E\big[h_{\tau_{l+1}}(\bv)(1)-h_{\tau_{l}}(\bv)(1) ~\big |~\mathscr{F}_{T_{l}}\big]=0
\end{align*}
almost surely,
which completes the proof.
\end{proof}

Now let $ \bu ,\bv \in \mathscr{V} $ with $ \bu(2)=\bv(2) $, and  without loss of generality, assume that $ \bu(1)>\bv(1) $. Since the paths 
starting from $ \bu $ and $ \bv $ do not cross each other, the martingale $ \{Z_{l}(\bu ,\bv):l \geq 0\} $ is non-negative. Therefore, by the martingale convergence theorem, there exists a random variable $ Z_{\infty} $ such that
$Z_{l}(\bu ,\bv)  \xrightarrow{a.s.} Z_{\infty}$ as $ l \to \infty$.
Since the Markov chain $ \{Z_{l}(\bu ,\bv):l \geq 0\} $ has $ 0 $ as its only absorbing state, we must have $ Z_{\infty}=0 $ almost surely. Hence, there exists some $ t \geq 0 $ so that $ Z_{l}(\bu ,\bv)=0  $ for all $ l\geq t $ almost surely. 

Next, if $ \bu(2)<\bv(2) $, by the Borel-Cantelli lemma,   we can find two vertices $ \bu^\prime ,\bv^\prime\in\mathscr{V} $ so that $ \bu^\prime(2)=\bv^\prime(2)=\bv(2) $ and
\begin{align*}
&\bu^\prime(1)<\bv(1)-\big(|\bu(1)-\bv(1)|+\bv(2)-\bu(2)\big),\\&
\bv^\prime(1)>\bv(1)+\big(|\bu(1)-\bv(1)|+\bv(2)-\bu(2)\big),
\end{align*}
almost surely. By the non-crossing property of our paths, the paths starting from $ \bu $ and $ \bv $ have to lie between the two paths starting from $ \bu^\prime $ and $ \bv^\prime $ from time $ \bv(2) $ onwards. Since the paths from $ \bu^\prime$ and $\bv^\prime$ must meet  almost surely, so do all paths sandwiched between them. This completes the proof of Theorem \ref{tree}(i) for $d=2$.

\vspace*{5mm}
\noindent \textit{Dimension 3.}
Let  $ \bu , \bv \in \Z^{d} $ be two distinct vertices  with $ \bu(d) = \bv(d) $, and let
$ \{U_{\bw}^{\bu}:\bw\in\Z^{d}\} $ and $ \{U_{\bw}^{\bv}:\bw\in\Z^{d}\} $ be two  independent collections of i.i.d. uniform $ (0,1) $ random variables. We  construct two independent paths $ \{\h_{n}(\bu):n \geq 1\} $  from $ \bu $  and  $ \{\h_{n}(\bv):n \geq 1\} $ from $ \bv $, using $ \{U_{\bw}^{\bu}:\bw\in\Z^{d}\} $ and $ \{U_{\bw}^{\bv}:\bw\in\Z^{d}\} $  respectively. Denoting the $ l $-th simultaneous renewal time of these two independent paths by $ \T_{l}(\bu ,\bv) $, we note that  $ \{\T_{l+1}(\bu ,\bv)-\T_{l}(\bu ,\bv):l \geq 0\} $ is a sequence of i.i.d. positive integer valued random variables. Also for any simultaneous regeneration time $ \T_{l}(\bu ,\bv) $, there exist non-negative integer valued random variables $ N_{l}(\bu) $ and $ N_{l}(\bv) $, each of which are functions of both $ \bu $ and $ \bv $, so that
\[\T_{l}(\bu ,\bv)=\h_{N_{l}(\bu)}(\bu)(d)=\h_{N_{l}(\bv)}(\bv)(d).\]
Taking
\begin{align*}
&R_{n}^{(\bx)}:=\inf\{\h_{k}(\bx)(d)-\bx(d):k \geq 1 ,\h_{k}(\bx)(d)-\bx(d)\geq n\}-n,~~~~\bx\in\{\bu ,\bv\},
\end{align*}
we have 
\begin{align*}
\T_{1}(\bu ,\bv)-\T_{0}(\bu ,\bv):=\inf\{n \geq 1 : R_{n}^{(\bu)}=R_{n}^{(\bv)}=0\}.
\end{align*}

Now we invoke the following result.

 \begin{lemma}\textup{\textbf{(Lemma 3.2 of \cite{Roy2016})}}
Let $ \{\xi_{n}^{(1)}:n\geq 1\} $ and $ \{\xi_{n}^{(2)}:n\geq 1\} $ be two independent collections of i.i.d. positive integer valued random variables with $ \max\{\P(\xi_{1}^{(1)}\geq m),\P(\xi_{1}^{(2)}\geq m)\}\leq C_{5,1}\exp\{-C_{6,1}m\} $ for any $ m\geq 1 $ and  positive constants $ C_{5,1},C_{6,1} $. Also, let $ \min\{\P(\xi_{1}^{(1)}=1),\P(\xi_{1}^{(2)}=1)\}>0 $. For $ i=1,2 $ and $ k,n\geq 1 $,  let $ S_{k}^{(i)}:=\sum_{j=1}^{k}\xi_{j}^{(i)} $ and
 $R_{n}^{(i)}:=\inf\{S_{k}^{(i)}:k\geq 1 ,S_{k}^{(i)}\geq n\}-n$.
Taking $ \upsilon^{R}:=\inf\{n\geq 1 :R_{n}^{(1)}=R_{n}^{(2)}=0\} $,  we have
\begin{align*}
\P\{\upsilon^{R}\geq m\}\leq C_{5,2}\exp\{-C_{6,2}m\},~~~~m\geq 1 ,
\end{align*}
where $ C_{5,2} $ and $ C_{6,2} $ are some positive constants, depending only on the distributions of $ \xi_{n}^{(i)} $'s.
\end{lemma}

From the lemma, for $ m\geq 1 $,
\begin{align}
\P\big\{\h_{N_{l+1}(\bu)}(\bu)(d)-\h_{N_{l}(\bu)}(\bu)(d)\geq m\big\}&=\P\big\{\h_{N_{l+1}(\bv)}(\bv)(d)-\h_{N_{l}(\bv)}(\bv)(d)\geq m\big\}\nonumber\\
&=\P\big\{\T_{l+1}(\bu ,\bv)-\T_{l}(\bu ,\bv)\geq m\big\}\nonumber\\
&\leq C_{5}\exp\{-C_{6}m\},\label{c5c6}
\end{align}
where $ C_{5} $ and $ C_{6} $ are positive constants depending only on the distribution of $ (\h_{n}(\bu)(d)-\h_{n-1}(\bu)(d)) $'s.

Now we study the displacement in the first $d-1$ coordinates. For $ \bw=(\bw(1),\ldots ,\bw(d)) $, we denote $ \bar{\bw}=(\bw(1),\ldots ,\bw(d-1)) $. For $ l\geq 1 $, let 
 \begin{align*}
&\psi_{l}^{\bu}=\psi_{l}^{\bu}(\bu ,\bv):=\hh_{N_{l}(\bu)}(\bu)-\hh_{N_{l-1}(\bu)}(\bu)=\sum_{t=N_{l-1}(\bu)+1}^{N_{l}(\bu)}\big [ \hh_{t}(\bu)-\hh_{t-1}(\bu)  \big],\\
& \psi_{l}^{\bv}=\psi_{l}^{\bv}(\bu ,\bv):=\hh_{N_{l}(\bv)}(\bv)-\hh_{N_{l-1}(\bv)}(\bv)=\sum_{t=N_{l-1}(\bv)+1}^{N_{l}(\bv)}\big [ \hh_{t}(\bv)-\hh_{t-1}(\bv)  \big].
\end{align*}

For $ d=3 $, as in the Appendix of \cite{Roy2016} (page 1141) we have that if $ \bar{\bu}-\bar{\bv}=\bx\in\Zz $, then
\begin{align}
&\E\big[\Vert(\bar{\bu}+\psi_{1}^{\bu})-(\bar{\bv}+\psi_{1}^{\bv}) \Vert_{2}^{2}-\Vert \bx\Vert_{2}^{2} \big]=\alpha ,\label{4-2}\\
&\E\big[\big(\Vert(\bar{\bu}+\psi_{1}^{\bu})-(\bar{\bv}+\psi_{1}^{\bv}) \Vert_{2}^{2}-\Vert \bx\Vert_{2}^{2} \big)^{2}\big]\geq 2\alpha \Vert \bx\Vert_{2}^{2},\label{4-3}\\
&\E\big[\big(\Vert(\bar{\bu}+\psi_{1}^{\bu})-(\bar{\bv}+\psi_{1}^{\bv}) \Vert_{2}^{2}-\Vert \bx\Vert_{2}^{2} \big)^{3}\big]=O(\Vert \bx\Vert_{2}^{2}) ~~~\text{as~}\Vert \bx\Vert_{2}\rightarrow\infty , \label{4-4}
\end{align}
where $ \alpha $ is some non-negative constant.

Now consider two distinct open vertices $ \bu,\bv\in\Z^{3} $ and  assume that $ \bu(3)=\bv(3) $. We will apply the Foster-Lyapunov criterion (see Proposition 5.3 in Chapter I of \cite{Asmussen})  on the process $ \{Z_{l}(\bu,\bv):l \geq 0\} $ where $Z_{l}(\bu,\bv)$ is as given in (\ref{rZ}).
\begin{proposition}\textup{\textbf{(Foster-Lyapunov Criterion) }}
An irreducible Markov chain with state space $ E \subseteq \Z^{d} $ and stationary transition probability matrix $ (p_{ij})_{i,j\in E} $ is recurrent if there exists a function $ f : E \rightarrow \R $ such that $ f(x)\rightarrow \infty $ as $\Vert x\Vert_{d}\rightarrow\infty$ and
\[ \sum_{k\in E}p_{jk}f(k)\leq f(j) ~~~~~~~~~  \text{for} ~ j\in E_{0},\]
 where $ E_{0} $ is a subset of  $  E $ so that $ E\setminus E_{0} $ is finite.
\end{proposition}

The Markov chain $ \{Z_{l}(\bu,\bv):l \geq 0\} $ is not irreducible;   the  state $ (0,0)$ being absorbing. Because of this, we modify the Markov process so that it has the same transition probabilities as $ \{Z_{l}(\bu,\bv):l \geq 0\} $ except that instead of $ (0,0) $ being an absorbing state, it goes to $ (1,0) $ with probability 1. With a slight abuse of notation, we denote the modified Markov chain by $ \{Z_{l}(\bu,\bv):l \geq 0\} $ again. Using the Foster-Lyapunov criterion, we show that the Markov chain $ \{Z_{l}(\bu,\bv):l \geq 0\} $ is recurrent. This will prove that the graph $ \mathscr{G} $ is connected. 

To apply the Foster-Lyapunov criterion, consider $ f:\Zz\rightarrow [0,\infty) $ by 
\[ f(\bx)=\sqrt{\ln(1+\Vert \bx\Vert_{2}^{2})} .\]
Also, define a function $ g:[0,\infty) \rightarrow[0,\infty)  $ by $ g(t)=\sqrt{\ln(1+t)} $. In this case, $ f(\bx)=g(\Vert \bx\Vert_{2}^{2}) $. The fourth derivative of $ g $ is  non-positive. Therefore, using Taylor's expansion, we conclude
\begin{align}
&\E\big[f(Z_{1}(\bu ,\bv)) ~\big |~ Z_{0}(\bu ,\bv)=\bx\big]-f(\bx)\nonumber\\
&~=\E\big[g(\Vert Z_{1}(\bu ,\bv)\Vert_{2}^{2})-g(\Vert \bx\Vert_{2}^{2}) ~\big |~ Z_{0}(\bu ,\bv)=\bx\big]\nonumber\\
&~\leq \sum_{k=1}^{3}\dfrac{g^{(k)}(\Vert \bx\Vert_{2}^{2})}{k!}\E\big[\big(\Vert Z_{1}(\bu ,\bv)\Vert_{2}^{2}-\Vert \bx\Vert_{2}^{2}\big)^{k} ~\big |~ Z_{0}(\bu ,\bv)=\bx\big].\label{usetayorbound}
\end{align}
If the paths starting from $ \bu $ and $ \bv $ are independent until their first simultaneous regeneration time, then  
$Z_{1}(\bu ,\bv)$ has the same distribution as $ (\bar{\bu}+\psi_{1}^{\bu})-(\bar{\bv}+\psi_{1}^{\bv}) $. Although we do not have such independence we  couple the joint process and the independent processes to obtain a relation between the moments of $ \Vert Z_{1}(\bu ,\bv)\Vert_{2}^{2}-\Vert \bx\Vert_{2}^{2} $ and $ \Vert(\bar{\bu}+\psi_{1}^{\bu})-(\bar{\bv}+\psi_{1}^{\bv}) \Vert_{2}^{2}-\Vert \bx\Vert_{2}^{2} $ as follows:
\begin{proposition}
For any $ \bx\in\Zz\setminus\{(0,0)\} $ and $ k \geq 1 $, we have 
\begin{align*}
&\E \big[\big(\Vert Z_{1}(\bu ,\bv)\Vert_{2}^{2}-\Vert \bx\Vert_{2}^{2}\big)^{k} ~\big |~ Z_{0}(\bu ,\bv)=\bx\big]\\
&~\leq\E\big[\big(\Vert(\bar{\bu}+\psi_{1}^{\bu})-(\bar{\bv}+\psi_{1}^{\bv}) \Vert_{2}^{2}-\Vert \bx\Vert_{2}^{2} \big)^{k}\big] + C_{7}^{(k)} \Vert \bx\Vert_{2}^{2k}\exp\{-C_{8} \Vert \bx\Vert_{2}\},
\end{align*}
where $ C_{7}^{(k)} $ and $ C_{8} $ are some positive constants depending on the distribution of $ (\h_{n}(\bu)(3)-\h_{n-1}(\bu)(3)) $'s, and $ C_{7}^{(k)} $ depends on  $ k $ too.
\end{proposition}
\begin{proof}
By the translation invariance of our model, it suffices to prove the result for $ \bu=(\bx(1),\bx(2),0) $ and $ \bv=(0,0,0) $. 
Let $ r=\Vert \bar{\bu}-\bar{\bv}\Vert_{1}/3=(\vert \bx(1)\vert+\vert \bx(2)\vert)/3 $. Recall that to construct the independent paths, we use the collections  $ \{U_{\bw}^{\bu}:\bw\in\Z^{3}\} $ and $ \{U_{\bw}^{\bv}:\bw\in\Z^{3}\} $. We now consider another  collection of i.i.d. uniform $ (0,1) $ random variables $ \{U^{\prime}_{\bw}:\bw\in\Z^{3}\} $,  independent of all other random variables, and define a new collection of uniform random variables $ \{\tilde{U}_{\bw}:\bw\in\Z^{3}\} $ by
\begin{align*}
\tilde{U}_{\bw}:= \left\{
\begin{array}{rl}
U_{\bw}^{\bu}, & \text{if } \bw\in V(\bu ,\lfloor r\rfloor),\\
U_{\bw}^{\bv}, & \text{if } \bw\in V(\bv ,\lfloor r\rfloor),\\
U^{\prime}_{\bw}, & \text{otherwise}. 
\end{array} \right.
\end{align*}
Using this collection, we construct the joint path $ \{(h_{n}(\bu),h_{n}(\bv)):n \geq 0\} $ starting from  $ (\bu,\bv) $ until their first simultaneous regeneration time $ T_{1} $ at step $ \tau_{1} $. We observe that 
\[\Vert Z_{1}(\bu,\bv)\Vert_{2}\leq\Vert Z_{1}(\bu,\bv)-\bar{\bu}\Vert_{1}+\Vert\bar{\bu}\Vert_{2}\leq 4h_{\tau_{1}}(\bu)(3)+\Vert\bx\Vert_{2},\]
and
\[ \Vert \psi_{1}^{\bu}-\psi_{1}^{\bv}\Vert_{2}\leq \Vert \psi_{1}^{\bu}\Vert_{1}+\Vert\psi_{1}^{\bv}\Vert_{1}\leq 4\h_{N_{1}(\bu)}(\bu)(3) .\]
Now, define the event 
\[A(r):=\{\h_{N_{1}(\bu)}(\bu)(3)>r\}.\]
 An argument as in Proposition \ref{ppppp} yields that  $ \h_{N_{1}(\bu)}(\bu)(3) $, like $ h_{\tau_{1}}(\bu)(3) $, has exponentially decaying tail probabilities with a tail bound which does not depend on $ \bu $ or $ \bv $. Hence, we have
\begin{align*}
&\E \Big[\big(\Vert Z_{1}(\bu ,\bv)\Vert_{2}^{2}-\Vert \bx\Vert_{2}^{2}\big)^{k}-\big(\Vert(\bar{\bu}+\psi_{1}^{\bu})-(\bar{\bv}+\psi_{1}^{\bv}) \Vert_{2}^{2}-\Vert \bx\Vert_{2}^{2} \big)^{k}\Big] \\
&~=\E \Big[\Big(\big(\Vert Z_{1}(\bu ,\bv)\Vert_{2}^{2}-\Vert \bx\Vert_{2}^{2}\big)^{k}-\big(\Vert(\bar{\bu}+\psi_{1}^{\bu})-(\bar{\bv}+\psi_{1}^{\bv}) \Vert_{2}^{2}-\Vert \bx\Vert_{2}^{2} \big)^{k}\Big)\mathbf{1}(A(r))\Big] \\
&~\leq2^{k}\E \Big[\Big(\Vert Z_{1}(\bu ,\bv)\Vert_{2}^{2k}+(2^{2k}+2)\Vert \bx\Vert_{2}^{2k}+2^{2k}\Vert \psi_{1}^{\bu}-\psi_{1}^{\bv} \Vert_{2}^{2k}\Big)\mathbf{1}(A(r))\Big] \\
&~\leq2^{7k}\Big(\E\Big[\big(h_{\tau_{1}}^{2k}(\bu)(3)+\Vert \bx\Vert_{2}^{2k}+(\h_{N_{1}(\bu)}(\bu)(3))^{2k}\big)^{2}\Big]\P\big\{\h_{N_{1}(\bu)}(\bu)(3)>r\big\}\Big)^{\frac{1}{2}} \\
&~\leq2^{7k}\times 3\Big(\E [h_{\tau_{1}}^{4k}(\bu)(3)]+\E[(\h_{N_{1}(\bu)}(\bu)(3))^{4k}]\Big)^{\frac{1}{2}}
\Vert \bx\Vert_{2}^{2k}C_{5}^{\frac{1}{2}}\exp\{-C_{6}\Vert \bx\Vert_{2}/6\},
\end{align*}
where we have used Cauchy-Schwarz inequality in the penultimate line and inequality \eqref{c5c6} in the last line above. This establishes the proposition.
\end{proof}

We now return to the relation \eqref{usetayorbound}. The above proposition implies that
\begin{align*}
&\E\big[f(Z_{1}(\bu ,\bv)) ~\big |~ Z_{0}(\bu ,\bv)=\bx\big]-f(\bx)\\
&~\leq \sum_{k=1}^{3}\dfrac{g^{(k)}(\Vert \bx\Vert_{2}^{2})}{k!}\E\big[\big(\Vert(\bar{\bu}+\psi_{1}^{\bu})-(\bar{\bv}+\psi_{1}^{\bv}) \Vert_{2}^{2}-\Vert \bx\Vert_{2}^{2} \big)^{k}\big]\\
&~+ \sum_{k=1}^{3}\dfrac{g^{(k)}(\Vert \bx\Vert_{2}^{2})}{k!} C_{7}^{(k)} \Vert \bx\Vert_{2}^{2k}\exp\{-C_{8} \Vert \bx\Vert_{2}\}.
\end{align*}
For $ \Vert \bx\Vert_{2} $ large enough, from \eqref{4-2}, \eqref{4-3} and \eqref{4-4} we  have
 \begin{align*}
&\sum_{k=1}^{3}\dfrac{g^{(k)}(\Vert \bx\Vert_{2}^{2})}{k!}\E\big[\big(\Vert(\bar{\bu}+\psi_{1}^{\bu})-(\bar{\bv}+\psi_{1}^{\bv}) \Vert_{2}^{2}-\Vert \bx\Vert_{2}^{2} \big)^{k}\big]\\&~
\leq \big[-\alpha\Vert \bx\Vert_{2}^{2}\big(\ln(1+\Vert \bx\Vert_{2}^{2})\big)^{-3/2}\big]/\big(8(1+\Vert \bx\Vert_{2}^{2})^{2}\big).
\end{align*}
Also,
 \begin{align*}
 &\sum_{k=1}^{3}\dfrac{g^{(k)}(\Vert \bx\Vert_{2}^{2})}{k!} C_{7}^{(k)} \Vert \bx\Vert_{2}^{2k}\exp\{-C_{8} \Vert \bx\Vert_{2}\}\\&~
\leq 2\max\{C_{7}^{(1)},C_{7}^{(3)}\}\exp\{-C_{8} \Vert \bx\Vert_{2}\}.
\end{align*}
Therefore, 
\begin{align*}
\E\big[f(Z_{1}(\bu ,\bv)) ~\big |~ Z_{0}(\bu ,\bv)=\bx\big]-f(\bx)\leq 0,
\end{align*}
for $ \Vert \bx\Vert_{2} $ large. This completes the proof of Theorem \ref{tree}(i) for dimension 3.

\subsection{Proof of Theorem \ref{tree}(ii)}
Suppose $ d\geq 4 $. By the ergodicity inherent in our model to establish Theorem \ref{tree}(ii) it is suffices to show
\begin{align}\label{a-dim4}
 \P\{\mathscr{G}\text{ has at least }m\text{ distinct trees}\}>0 \text{ for all } m\geq 2.
\end{align} 
 The proof is adapted from  \cite{Gangopadhyay} although for our model, we have to use the regeneration process. In this it is also different from that of \cite{Athreya}. The idea of the proof is to study the paths from two vertices $\bu$ and $\bv$ between two successive regeneration steps. Suppose at a stage $k\geq 1$, having obtained $h_k(\bu)$ and $h_k(\bv)$ with $h_k(\bu)$ below $h_k(\bv)$ (say), the region explored by $h_k(\bu)$  to obtain $h_{k+1}(\bu)$ has empty intersection with the history set $\Delta_k$ present at that stage. We show that this happens with a positive probability for every stage $k$. Hence for all practical purposes, the path obtained from $\bu$ is independent of the path obtained from $\bv$. Thinking of 
$\{h_{\tau_l}(\bu) : l \geq 1\}$ and $\{h_{\tau_l}(\bv) : l \geq 1\}$ as two random walks in the first $d-1$ coordinates, we see that with a positive probability, the incremental displacements $h_{\tau_{l+1}}(\bu) - h_{\tau_{l}}(\bu)$ and $h_{\tau_{l+1}}(\bv) - h_{\tau_{l}}(\bv)$ in the first $d-1$ coordinates are independent and identically distributed for $l \geq 1$. Hence, invoking the transience of random walks for 3 or higher dimensions, we have that, with a positive probability, the paths from $\bu$ and $\bv$ do not meet.

The first step towards this is the following proposition which uses the independence of the process between different joint renewal times of the $m$ paths.
\begin{proposition}\label{prop4.3}
For $ 0<\varepsilon<1/3 $, there exist positive constants $ C_{9} $, $ \beta=\beta(\varepsilon) $  and $ n_{0}\geq 1  $ such that
\begin{align*}\label{prop4.3formula}
\inf_{(\bu,\bv)\in A_{n,\varepsilon }}\P \big\{Z_{n^{4}}(\bu ,\bv)\in D_{n^{2(1+\varepsilon)}}\setminus D_{n^{2(1-\varepsilon)}}~\big |~ \bu,\bv\in\mathscr{V}\big\}\geq 1-C_{9}n^{-\beta}
\end{align*}
for all $ n\geq n_{0} $, where $ D_{r}:=\{\bw\in\Z^{d-1}:\Vert \bw\Vert_{1}\leq r\} $ and $ A_{n,\varepsilon }:=\{(\bu,\bv)\in\Z^{d}\times\Z^{d}:\bu(d)=\bv(d),\bar{\bu}-\bar{\bv}\in D_{n^{1+\varepsilon}}\setminus D_{n^{1-\varepsilon}}\} $.
\end{proposition}
\begin{proof}
Fix $ 0<\varepsilon<1/3 $ and  $ n \geq 1$. Consider the independent paths starting from  $ \bu ,\bv\in\Z^{d} $ with $ \bu(d)=\bv(d) $, constructed using the collections $ \{U_{\bw}^{\bu}:\bw\in\Z^{d}\} $ and $ \{U_{\bw}^{\bv}:\bw\in\Z^{d}\} $ respectively, and define
\begin{align*}
W_{n,\varepsilon}(\bu ,\bv):=&\Big\{\hh_{N_{n^{4}}(\bu)}(\bu)-\hh_{N_{n^{4}}(\bv)}(\bv)\in D_{n^{2(1+\varepsilon)}}\setminus D_{n^{2(1-\varepsilon)}},\\
&~~\hh_{N_{j}(\bu)}(\bu)-\hh_{N_{j}(\bv)}(\bv)\notin D_{K\ln n}\text{ for all } j=0,1,\ldots ,n^{4}
\Big\},
\end{align*}
where $ K $ is a  positive constant to be  specified later. 
 Now, we use the following lemma:
\begin{lemma}\label{GRS} \textup{\textbf{(Lemma 3.3 of \cite{Gangopadhyay})}} 
For $ 0<\varepsilon<1/3 $, there exist $ \alpha=\alpha(\varepsilon)>0 $ and positive constants $ C_{10,1},C_{10,2},C_{10,3} $ such that for all $ n $ sufficiently large, we have
\begin{description}
\item[(a)] $ \sup_{(\bu,\bv)\in A_{n,\varepsilon }}\P\big\{\hh_{N_{n^{4}}(\bu)}(\bu)-\hh_{N_{n^{4}}(\bv)}(\bv)\notin D_{n^{2(1+\varepsilon)}}\big\}\leq C_{10,1}n^{-\alpha} $,
\item[(b)] $ \sup_{(\bu,\bv)\in A_{n,\varepsilon }}\P\big\{\hh_{N_{n^{4}}(\bu)}(\bu)-\hh_{N_{n^{4}}(\bv)}(\bv)\in D_{n^{2(1-\varepsilon)}}\big\}\leq C_{10,2}n^{-\alpha} $,
\item[(c)] $ \sup_{(\bu,\bv)\in A_{n,\varepsilon }}\P\big\{\hh_{N_{j}(\bu)}(\bu)-\hh_{N_{j}(\bv)}(\bv)\in D_{K\ln n}\text{ for some }j=0,1,\ldots ,n^{4}\big\} $

$\qquad \qquad \qquad\qquad \qquad \qquad\qquad \qquad \qquad\qquad \qquad \qquad\leq C_{10,3}n^{-\alpha}$.
\end{description}
\end{lemma}

From the lemma, there exists $ n_{0} $ such that for all $ n\geq n_{0} $,
\begin{equation}
\label{prop4.4}
\inf_{(\bu,\bv)\in A_{n,\varepsilon }}\P \big\{W_{n,\varepsilon}(\bu ,\bv) ~\big |~ \bu,\bv\in\mathscr{V}\big\}\geq 1-C_{10}n^{-\alpha},
\end{equation}
where $ C_{10} $ and $ \alpha=\alpha(\varepsilon) $ are some positive constants.

For $ l\geq 1 $, let $ r_{l}(n):=\min\{\lfloor \frac{K\ln n}{2d}\rfloor, \h_{N_{l}(\bu)}(\bu)(d)-\h_{N_{l-1}(\bu)}(\bu)(d)\} $. On the event $ W_{n,\varepsilon}(\bu ,\bv) $, we have
\[ \Big[\bigcup_{l=1}^{n^{4}}V\big(\h_{N_{l-1}(\bu)}(\bu),r_{l}(n)\big)\Big]\cap \Big[\bigcup_{l=1}^{n^{4}}V\big(\h_{N_{l-1}(\bv)}(\bv),r_{l}(n)\big)\Big]=\emptyset , \]
so we  consider another independent collection of i.i.d. uniform $ (0,1) $ random variables $ \{U^{\prime\prime}_{\bw}:\bw\in\Z^{d}\} $ and define a new family  $ \{\check{U}_{\bw}:\bw\in\Z^{d}\} $ as
\begin{align*}
\check{U}_{\bw}:=\left\{
\begin{array}{rl}
U_{\bw}^{\bu} & \text{if }\bw\in \bigcup_{l=1}^{n^{4}}V\big(\h_{N_{l-1}(\bu)}(\bu),r_{l}(n)\big),\\
U_{\bw}^{\bv} & \text{if }\bw\in \bigcup_{l=1}^{n^{4}}V\big(\h_{N_{l-1}(\bv)}(\bv),r_{l}(n)\big),\\
U_{\bw}^{\prime\prime} & \text{otherwise}. 
\end{array} \right.
\end{align*}
Taking
\[B_{l}(n):=\big\{\h_{N_{l}(\bu)}(\bu)(d)-\h_{N_{l-1}(\bu)}(\bu)(d)<  \tfrac{K\ln n}{2d}\big\}.\]
we see that, if $ B_{l}(n) $ occurs for each $ l=1,\ldots ,n^{4} $ then, for $Z_n(\bu,\bv)$ obtained from the joint paths $ \{(h_{n}(\bu),h_{n}(\bv)):n \geq 0\} $ constructed using $ \{\check{U}_{\bw}:\bw\in\Z^{d}\} $,
we have 
\[Z_{n^{4}}(\bu ,\bv)=\hh_{N_{n^{4}}(\bu)}(\bu)-\hh_{N_{n^{4}}(\bv)}(\bv).\]
Therefore, from \eqref{c5c6} and \eqref{prop4.4}  we get, for all $ n\geq n_{0} $ and $ (\bu,\bv)\in A_{n,\varepsilon } $,
\begin{align*}
&\P \big\{Z_{n^{4}}(\bu ,\bv)\in D_{n^{2(1+\varepsilon)}}\setminus D_{n^{2(1-\varepsilon)}} ~\big |~ \bu,\bv\in\mathscr{V}\big\}\\&~
\geq \P \big\{\big(\cap_{l=1}^{n^{4}}B_{l}(n)\big)\cap W_{n,\varepsilon}(\bu ,\bv)  ~\big |~ \bu,\bv\in\mathscr{V}\big\}\\&~
\geq  1-n^{4}C_{5}\exp\{-\tfrac{C_{6}K\ln n}{2d}\}-C_{10}n^{-\alpha}\\&~
\geq  1-C_{9}n^{-\beta},
\end{align*}
for $ K>\frac{8d}{C_{6}} $ and suitable choices of $ \beta ,C_{9} >0 $, which proves the proposition. 
\end{proof}

The rest of the proof of Theorem \ref{tree}(ii) follows along the lines as in Section 3.3 of \cite{Gangopadhyay}, and so we only provide a sketch of the proof. Using the notation of Proposition \ref{prop4.3}, taking $ \bu=(0,0,\ldots ,0)\in\Z^{d}  $, $ \bv=(n_{0},0,\ldots ,0)\in\Z^{d} $ (so $ (\bu ,\bv)\in A_{n_{0},\varepsilon} $) and $ r_{k}=\sum_{i=1}^{k}(n_{0}^{2^{i-1}})^{4} $, by Proposition \ref{prop4.3} it can be shown that 
\begin{align*}
&\P\{Z_{r_{k}}(\bu ,\bv)\in D_{n_{0}^{2^{k}(1+\varepsilon)}} \setminus D_{n_{0}^{2^{k}(1-\varepsilon)}}\text{ for all } k=1,\ldots ,j ~\big |~ \bu ,\bv\in\mathscr{V}\}\\&\geq \prod _{k=1}^{j}\big [1-C_{9}(n_{0}^{2^{k-1}})^{-\beta}\big].
\end{align*} 
Hence, since $ (0,\ldots ,0)\in\Z^{d-1} $ is the absorbing state of process $ \{Z_{l}(\bu ,\bv):l\geq 0\} $,
\begin{align*}
&\P\big\{\mathscr{G}\text{ has at least two distinct trees}\big\}\\
&~\geq\P\big\{Z_{l}(\bu ,\bv)\neq (0,\ldots ,0)\text{ for all } l\geq 0 ~\big |~ \bu ,\bv\in\mathscr{V}\big\}\times\P\{\bu ,\bv\in\mathscr{V}\}\\&
~\geq p^{2}\prod _{k=1}^{\infty}\big [1-C_{9}(n_{0}^{2^{k-1}})^{-\beta}\big]>0,
\end{align*}
and \eqref{a-dim4} is established for $ m=2 $. To establish \eqref{a-dim4} for $ m\geq 3 $, it is enough to take 
$\bu^{i}=((i-1)n_{1},0,\ldots ,0)\in\Z^{d}$ for each $i=1,\ldots ,m $,
where $ n_{1}\geq 1 $ is some constant satisfying $ n_{1}\geq\max\{n_{0},m^{1/\varepsilon}\} $ and
$\prod _{k=1}^{\infty}\big [1-C_{9}(n_{1}^{2^{k-1}})^{-\beta}\big]>1-\delta,$
for some $ \delta>0 $ that $ m(m-1)\delta/2 <1 $. Because then $ \P\{Z_{l}(\bu^{i} ,\bu^{j})\neq(0,\ldots ,0)\text{ for all } l\geq 0 ~\big |~ \bu^{i},\bu^{j}\in\mathscr{V}\}>1-\delta $ for each $ i,j\in\{1,\ldots ,m\} $ with $ i>j $, and hence $ \P\{\mathscr{G}\text{ has at least m distinct trees}\}> p^{m}\big(1-m(m-1)\delta /2 \big)>0 $. 
This completes the proof.

 Finally, the proof of Theorem \ref{tree}(iii) follows from a similar Burton-Keane type argument \cite{Burton} as was used in Theorem 2 of \cite{Athreya}.

\section{Convergence to the Brownian Web}\label{s4}
In this section we prove Theorem \ref{BW}. For a subset of paths $ \Gamma $ in $ \Pi $, and $ t_{0} ,t, a,b\in\R $ with $ t > 0 $ and $ a < b $, consider the following counting random variables:
\begin{align*}
&\eta_{\Gamma}(t_{0},t;a,b):=\#\big\{\pi(t_{0}+t):\pi\in\Gamma , \varsigma_{\pi}\leq t_{0},\pi (t_{0})\in [a,b]\big\},
\\&\hat{\eta}_{\Gamma}(t_{0},t;a,b):=\#\big\{\pi(t_{0}+t):\pi\in\Gamma , \varsigma_{\pi}\leq t_{0},\pi (t_{0}+t)\in [a,b]\big\}.
\end{align*}
From now on we denote the Brownian motion with unit diffusion constant starting from $ \bx $ by $ B_{\bx} $ and coalescing Brownian motions with unit diffusion constants starting from $ \bx_{1},\ldots ,\bx_{k} $ by $ (W_{\bx_{1}},\ldots ,W_{\bx_{k}}) $. For the proof of Theorem \ref{BW}, we use the following theorem which is Theorem 2.2 of \cite{Fontes2}:

\begin{theorem}\label{barayehamg}
Suppose $ \Theta_{1},\Theta_{2},\ldots $ are  $({\cal H},{{\cal B}_{\cal H}}) $ valued random variables with non-crossing paths. Assume that the following conditions hold:
\begin{itemize}
\item[{\rm {(I$_1$)}}] For all $ \by\in\R^{2} $, there exist $ \zeta_{n}^{\by}\in\Theta_{n} $ such that, for any finite set of points  $ \by_{1},\ldots ,\by_{k} $ from a deterministic countable dense set $ \mathcal{D} $ of $ \mathbb{R}^{2} $,  $ (\zeta_{n}^{\by_{1}},\ldots ,\zeta_{n}^{\by_{k}})$ converges in distribution to $(W_{\by_{1}},\ldots ,W_{\by_{k}}) $ as $ n\rightarrow \infty $;
\item[{\rm {(B$_1$)}}] For all $ t > 0 $, $ \limsup_{n\rightarrow\infty}\sup_{(a,t_{0})\in\mathbb{R}^{2}}\mathbb{P}\{\eta_{\Theta_{n}}(t_{0},t;a,a+\varepsilon)\geq 2\}\rightarrow 0 $ as $ \varepsilon\downarrow 0$;
 \item[{\rm {(B$_2$)}}] For all $ t > 0 $, $ \varepsilon^{-1}\limsup_{n\rightarrow\infty}\sup_{(a,t_{0})\in\mathbb{R}^{2}}\mathbb{P}\{\eta_{\Theta_{n}}(t_{0},t;a,a+\varepsilon)\geq 3\}\rightarrow 0 $ as $ \varepsilon\downarrow 0 $.
\end{itemize}
Then, $ \Theta_{n}$ converges in distribution to $\mathcal{W} $ as $n\to \infty$.
\end{theorem}

To prove Theorem \ref{BW}, we first show that $ \bar{\mathcal{X}}_{n}(\sigma ,\gamma) $ is compact for each $ n\geq 1 $ and then we show that the sequence $\{ \bar{\mathcal{X}}_{n}(\sigma ,\gamma) : n \geq 1\}$ satisfies the conditions of the above theorem.

\subsection{$\bar{\mathcal{X}}_{n}(\sigma ,\gamma)$ is an  $({\cal H},{{\cal B}_{\cal H}})$ valued random variable}

It suffices to show that  $ \mathcal{X} $ has compact closure in $ (\Pi ,d_{\Pi}) $.
For any path $ \pi\in\Pi $, define the extended path $ \hat{\pi} $ as follows:
\begin{align*}
\hat{\pi}(t):= \left\{
\begin{array}{rl}
\pi( \varsigma_{\pi}) & \text{for }t\leq  \varsigma_{\pi} ,\\
\pi(t)~ &  \text{for }t> \varsigma_{\pi} .
\end{array} \right.
\end{align*}
 For $ {\pi _1},{\pi _2}\in\Pi $, let
\begin{align*}
\hat{d}_{\Pi}({\hat{\pi }_1},{\hat{\pi} _2}):= |\tanh ({\varsigma _{{\pi _1}}}) - \tanh ({\varsigma _{{\pi _2}}})| \vee \mathop {\sup }\limits_{t \ge {\varsigma _{{\pi _1}}} \wedge {\varsigma _{{\pi _2}}} } \left| {\frac{{\tanh (\hat{\pi} _{1}(t ))}}{{1 + |t|}} - \frac{{\tanh (\hat{\pi} _{2}(t ))}}{{1 + |t|}}} \right|,
\end{align*}
so that $ \hat{d}_{\Pi}({\hat{\pi }_1},{\hat{\pi} _2})=d_{\Pi}({\pi _1},{\pi _2})  $. Thus we need to show that $ \hat{\mathcal{X}}:=\{\hat{\pi}:\pi\in\mathcal{X}\} $ has compact closure in $ (\hat{\Pi} ,\hat{d}_{\Pi}) $, where $ \hat{\Pi}:=\{\hat{\pi}:\pi\in\Pi\} $. For this purpose, we prove that the closure of $ f(\hat{\mathcal{X}}) $
is compact for some homeomorphism $ f $.

Note that each path $ \hat{\pi}\in \hat{\Pi} $ can be seen as the graph $ \{(\hat{\pi}(t),t):t\in\R \}\subseteq\R^{2} $. Taking a map $ (\varphi ,\psi ):\R^{2} \rightarrow (-1,1)\times (-1,1) $ as
 \begin{align*}
\big(\varphi ,\psi \big)(x,t)=\big(\varphi(x,t) ,\psi(t) \big):=\big(\dfrac{\tanh(x)}{1+\vert t\vert},\tanh(t)\big ),
 \end{align*}
 we define the  homeomorphism  $ f:\hat{\Pi}\rightarrow f(\hat{\Pi}) $ so that for $ \hat{\pi}\in\hat{\Pi} $, $ f(\hat{\pi}) $ is the following graph:
 \[\Big\{\big(\varphi(\hat{\pi}(t),t),\psi(t)\big):t\in\R \Big\}\subseteq(-1,1)\times(-1,1).\]
Now using Arzel\`{a}-Ascoli theorem and the Lipschitz continuity of the
paths, we prove that $ f(\hat{\mathcal{X}})\subseteq\big((-1,1)\times(-1,1)\big)^{\R } $ has  compact closure.
To show the equicontinuity of $ f(\hat{\mathcal{X}}) $ on $ \R $ we  note that every path $ f(\hat{\pi})\in f(\hat{\mathcal{X}}) $ is given by $ f(\hat{\pi})(t) = \big(\varphi(\hat{\pi}(t),t),\psi(t)\big) \in (-1,1)\times (-1,1)$.
We equip $ (-1,1)\times (-1,1) $  with the $ \mathcal{L}_{\infty} $-metric $ \hat{\rho} $  for which 
\begin{align*}
&\hat{\rho}\Big(\big(\varphi(\hat{\pi}_{1}(t_{1}),t_{1}),\psi(t_{1})\big),\big(\varphi(\hat{\pi}_{2}(t_{2}),t_{2}),\psi(t_{2})\big)\Big)\\&~=\big |\varphi(\hat{\pi}_{1}(t_{1}),t_{1})-\varphi(\hat{\pi}_{2}(t_{2}),t_{2})\big | \vee | \psi(t_{1})-\psi(t_{2}) |\\&~=\rho\big((\hat{\pi}_{1}(t_{1}),t_{1}),(\hat{\pi}_{2}(t_{2}),t_{2}) \big),
\end{align*}
for every $ \hat{\pi}_{1},\hat{\pi}_{2}\in\hat{\Pi} $, $ t_{1},t_{2}\in\R $ and where $\rho$ is as defined in \eqref{rhodef}. 
We  now show that,  for any $ \varepsilon>0 $, there exists $ \delta>0 $ such that if $ t_{1},t_{2}\in\R $ with $ |t_{1}-t_{2}|<\delta $, then 
\[\sup\Big\{\rho\big((\hat{\pi}(t_{1}),t_{1}),(\hat{\pi}(t_{2}),t_{2}) \big) : \hat{\pi}\in\hat{\mathcal{X}}\Big\}<\varepsilon ,\]
 which establishes the uniform equicontinuity  of $ f(\hat{\mathcal{X}}) $.  Indeed, if $ f(\hat{\mathcal{X}}) $ is not uniformly equicontinuous on  $ \R $, then there must exist $ \varepsilon>0 $ such that for all $ n\in\N $, we can find $ t^{n}_{1},t^{n}_{2}\in  \R $ with $ |t^{n}_{1}-t^{n}_{2}|<1/n $ and
\begin{align*}
\rho\big((\hat{\pi}_{n}(t^{n}_{1}),t^{n}_{1}),(\hat{\pi}_{n}(t^{n}_{2}),t^{n}_{2}) \big)>\varepsilon~~\text{for some}~\hat{\pi}_{n}\in\hat{\mathcal{X}}.
\end{align*}
However this is not possible because
$$
|\psi(t) -  \psi(s)| = |\tanh(t) -  \tanh(s)|  \to 0 \text{ as } |t-s| \to 0,
$$
and, noting that for our paths  $|\hat{\pi}(t) - \hat{\pi}(s)|\leq |t-s|$, we have
$$
\big|\varphi(\hat{\pi}(t),t) - \varphi(\hat{\pi}(s),s)\big|
= \left|\frac{\tanh(\hat{\pi}(t))}{1+|t|} -  \frac{\tanh(\hat{\pi}(s))}{1+|s|}\right| \to 0 \text{ as } |t-s| \to 0.
$$

In addition, for any $ t\in \R $, if $ f_{t}(\hat{\mathcal{X}}) $ is the set of postions of the paths in  $ f(\hat{\mathcal{X}}) $ at time $ t $ then
\begin{align*}
f_{t}(\hat{\mathcal{X}})=\big\{\big(\varphi(\hat{\pi}(t),t),\psi(t)\big):\hat{\pi}\in \hat{\mathcal{X}}\big\}=\big\{\varphi(\hat{\pi}(t),t):\hat{\pi}\in \hat{\mathcal{X}}\big\}\times \big\{\psi(t)\big\}.
\end{align*}
Since  $ \big\{\varphi(\hat{\pi}(t),t):\hat{\pi}\in \hat{\mathcal{X}}\big\} $ is bounded, the closure of $ \big\{\varphi(\hat{\pi}(t),t):\hat{\pi}\in \hat{\mathcal{X}}\big\} $ is compact. Thus, $ f_{t}(\hat{\mathcal{X}}) $ has compact closure.

\subsection{Verification of condition ($ \text{I}_{1} $)}
For $ k\geq 0 $ and $ \bu\in\Zz $, let $ h_k(\bu) =(x_k,t_k)$ (say) and
$X_{k+1}^{\bu} = x_{k+1}- x_k $, $ Y_{k+1}^{\bu} = t_{k+1} -t_k$  be the marginal increments  of $ \pi_{\bu}$ along each direction. We observe that $\{X_{k}^{\bu}: k \geq 1\} $ and $\{Y_{k}^{\bu}: k \geq 1\} $ are two collections of  i.i.d.  random variables, with $|X_{k}^{\bu}| \leq Y_{k}^{\bu}$ . Also $X_k^{\bu}$ has a symmetric distribution, with both $|X_{k}^{\bu}|$ and $Y_{k}^{\bu}$ having exponentially decaying tail probabilities; in particular, $ \P\{Y_{k}^{\bu}>m\}  \leq (1-p)^{(m+1)^{2}-1} $. 
Writing  $X_{ k}$ for $X_{ k}^{\mathbf{0}}$ and $Y_{ k}$ for $Y_{ k}^{\mathbf{0}}$, we have
\begin{align*}
h_k(\bu) = \big(\bu(1)+\sum_{i=1}^{k} X_{i}^{\bu}, \; \bu(2)+\sum_{i=1}^{k} Y_{i}^{\bu}\big),~~~~~~~ \\
 S_k:= \sum_{i=1}^{k} X_{i} \stackrel{d}{=} \sum_{i=1}^{k} X_{i}^{\bu}
\text{ and }  R_k := \sum_{i=1}^{k} Y_{i} \stackrel{d}{=} \sum_{i=1}^{k} Y_{i}^{\bu}. 
\end{align*}
 where we take $ \sum_{i=1}^{0}=0 $.

\begin{proposition}\label{I11}
There exist $ \sigma $ and $ \gamma $ such that $ \pi_{\mathbf{0}}^{(n)}(\sigma ,\gamma )  \Rightarrow B_{\mathbf{0}} $ in $ (\Pi ,d_{\Pi}) $.
\end{proposition}
\begin{proof}
 Let $ t\geq 0 $. Taking $ N(t):=\max\{n \geq 0:R_{n}\leq t\} $,  we have
 \[\pi_{\mathbf{0}}(t)=S_{N(t)}+\dfrac{t-R_{N(t)}}{R_{N(t)+1}-R_{N(t)}}X_{N(t)+1},\]
and its diffusively scaled version is
\begin{align*}
\pi_{\mathbf{0}}^{(n)}(\sigma ,\gamma )(t)=\dfrac{1}{n\sigma}\Big [\sum_{i=1}^{N(n^{2}\gamma t)}X_{i}+\dfrac{n^{2}\gamma t-R_{N(n^{2}\gamma t)}}{R_{N(n^{2}\gamma t)+1}-R_{N(n^{2}\gamma t)}}X_{N(n^{2}\gamma t)+1}\Big ].
\end{align*}
Taking $ \sigma =\sqrt{\V (X_{1})} $ and $ Z_{i}:=X_{i}/\sigma $,  we have
\[ \pi_{\mathbf{0}}^{(n)}(\sigma ,\gamma )(t)=\frac{1}{n}\Big [\sum_{i=1}^{N(n^{2}\gamma t)}Z_{i}+\dfrac{n^{2}\gamma t-R_{N(n^{2}\gamma t)}}{R_{N(n^{2}\gamma t)+1}-R_{N(n^{2}\gamma t)}}Z_{N(n^{2}\gamma t)+1}\Big ],\]
where $ Z_{i} $'s are  i.i.d. with mean 0 and variance 1. 
Next, we define another stochastic process $ \{\hat{\pi}_{n}(t):t\geq 0\} $ by
\begin{align*}
\hat{\pi}_{n}(t):=\dfrac{1}{n}\Big [\sum_{i=1}^{\lfloor n^{2}t\rfloor}Z_{i}+(n^{2}t-\lfloor n^{2}t\rfloor)Z_{\lfloor n^{2}t\rfloor+1}\Big ],
\end{align*}
and we see that
\begin{equation} \label{r2ndterm}
\pi_{\mathbf{0}}^{(n)}(\sigma ,\gamma )(t)\stackrel{d}{=}\hat{\pi}_{n}(N(n^{2}\gamma t)/n^{2})+\dfrac{1}{n}\Big [\dfrac{n^{2}\gamma t-R_{N(n^{2}\gamma t)}}{R_{N(n^{2}\gamma t)+1}-R_{N(n^{2}\gamma t)}}Z_{N(n^{2}\gamma t)+1}\Big ] .
\end{equation}
From Donsker's invariance principle, the process $ \hat{\pi}_{n} $ converges in distribution to the standard Brownian motion.  Also taking $ \gamma =\E[Y_{1}] $,  by the renewal theorem (see Theorem 4.4.1 of \cite{Durrett}), we have 
\[\dfrac{N(n^{2}\gamma t)}{n^{2}}\stackrel{a.s.}{\longrightarrow}\dfrac{\gamma t}{\E[Y_{1}]} = t \text{ as } n \to \infty. \]
Thus to complete the proof of the proposition, it suffices to show that the second term  in \eqref{r2ndterm} converges in probability to $0$ for this choice of $\gamma$.

Noting that
$0\leq\dfrac{n^{2}\gamma t-R_{N(n^{2}\gamma t)}}{R_{N(n^{2}\gamma t)+1}-R_{N(n^{2}\gamma t)}}\leq 1$, 
we have, for any $ \varepsilon >0 $ and $s > 0$, 
\begin{align*}
&\P\Big\{\sup_{0\leq t\leq s}\Big\vert\dfrac{n^{2}\gamma t-R_{N(n^{2}\gamma t)}}{R_{N(n^{2}\gamma t)+1}-R_{N(n^{2}\gamma t)}}Z_{N(n^{2}\gamma t)+1}\Big\vert >n\varepsilon\Big\}\\
&~\leq \P\big\{\sup_{i=1,\ldots ,\lfloor n^{2}\gamma s\rfloor +1}\vert Z_{i}\vert >n\varepsilon\big\}\\
&~\leq (\lfloor n^{2}\gamma s\rfloor +1) \P\{\vert X_{1}\vert >n\sigma\varepsilon\}\rightarrow 0  \text{ as } n \to \infty, 
\end{align*}
because $ \vert X_{1}\vert $ has exponentially decaying tail probabilities.  
\end{proof}

If $ (\frac{\bu_{n}(1)}{n\sigma},\frac{\bu_{n}(2)}{n^{2}\gamma}) \xrightarrow{\P} \bu$ as $n\rightarrow\infty $, noting that
$ \{(\pi_{\bu_{n}}(t),t):t\geq \bu_{n}(2) \}\disteq \bu_{n}+ \{(\pi_{\mathbf{0}}(t),t):t\geq 0 \}$, 
we have 
$\pi_{\bu_{n}}^{(n)}(\sigma ,\gamma )$ converges in distribution to $B_{\bu}$  as $n \to \infty$.

For $ \bx\in\R^{2} $ and $ n \geq 1 $, let
\begin{equation*}
\label{rxn}
\bx_{n}=\big(\lfloor n\sigma \bx (1)\rfloor +\mathtt{i}_{n} 
,\lfloor n^{2}\gamma \bx (2)\rfloor\big), 
\end{equation*}
where 
$\mathtt{i}_{n}= \mathtt{i}_{n}( \bx)  :=\inf \{i \geq 1 : (\lfloor n\sigma \bx (1)\rfloor +i,\lfloor n^{2}\gamma \bx (2)\rfloor ) \in\mathscr{V}\}$.
For any $ \delta\in(0,1) $, we have $\P\{\mathtt{i}_{n}> n^{\delta}\} \leq (1-p)^{n^{\delta}-1}$, so
$(\frac{\bx_{n}(1)}{n\sigma},\frac{\bx_{n}(2)}{n^{2}\gamma})\xrightarrow{\P} \bx$ as $n \to \infty$. Thus, taking $\zeta_{n}^{\bx}:=\pi^{(n)}_{\bx_{n}}(\sigma ,\gamma )$, we have that  $\zeta_{n}^{\bx}\Rightarrow B_{\bx}  $  as $n \to \infty$, which verifies condition ($ \text{I}_{1} $)  for $  k = 1 $.

For $ k\geq 2 $ we proceed by induction. Suppose condition ($ \text{I}_{1} $) holds for $ k-1 $ points.
Fix $ \bx^{1},\ldots ,\bx^{k}\in\R^{2} $ and without loss of generality assume that $ \bx^{1}(2)\leq\cdots\leq\bx^{k-1}(2)\leq\bx^{k}(2)=0 $. For simplicity in notation we write 
$\pi^{(n)}$ for $\pi^{(n)}(\sigma ,\gamma )$, where $\sigma$ and $\gamma$ are as in Proposition \ref{I11}.
Writing $\zeta_{n}^{\bx}:=\pi^{(n)}_{\bx_{n}}$ we now show that
$(\zeta_{n}^{\bx^{1}},\ldots ,\zeta_{n}^{\bx^{k}})\Rightarrow (W_{\bx^{1}},\ldots ,W_{\bx^{k}})$ as $n\to\infty$,
where the convergence occurs in the product space $ \Pi^{k} $ equipped with the metric
$d_{\Pi}^{(k)}((\pi_{1},\ldots ,\pi_{k}),(\theta_{1},\ldots ,\theta_{k}))=\sum_{i=1}^{k}d_{\Pi}(\pi_{i},\theta_{i})$.

The proof we present here is an adaptation of the argument
used by Roy \textit{et al.} \cite{Roy2016}.  The basic idea of the proof is noting that as long as the $ k $-th path is far from the other $ k - 1 $ paths, it can be approximated by an independent path, but when it comes close enough to one of the other $ k-1 $ paths, it quickly coalesces with the nearest path. As in \cite{Roy2016}, we first introduce a \textit{coalescence map}. Fix $ \alpha\in (0,\frac{1}{2}) $ and, for $ n\geq 1 $, let
$\mathcal{A}_{n,\alpha}:=\{(\pi_{1},\ldots ,\pi_{k})\in\Pi^{k}:\;t_{n,k} < \infty\}$  where 
\[t_{n,k}:=\inf \big\{t:t\geq\max \{ \varsigma_{\pi_{i}},\varsigma_{\pi_{k}}\},\vert\pi_{i}(t)-\pi_{k}(t)\vert\leq n^{-\alpha}
\text{~for~some~}1\leq i\leq k-1\big\}.\]
 Here, $ t_{n,k} $ is the time when the $k$-th path comes close to one of the other $ k-1 $ paths. In the case that $ t_{n,k}<\infty $, let $ s_{n,k}:=\frac{\lfloor n^{2}\gamma t_{n,k}\rfloor +1}{n^{2}\gamma} $ and
\begin{equation}
\label{ri0}
 i_{0}:=\min\big\{i\in\{1,\ldots , k-1\}:\vert\pi_{i}(t_{n,k})-\pi_{k}(t_{n,k})\vert\leq n^{-\alpha}\big\},
 \end{equation}
where $ i_{0} $ is the smallest index between indices of those paths which are close to $ \pi_{k} $ at time $ t_{n,k} $, and then define 
\begin{align*}
\bar{\pi}_{k}(t):= \left\{
\begin{array}{ll}
\pi_{k}(t) & \text{ for } \varsigma_{\pi_{k}}\leq t\leq t_{n,k},\\
\pi_{k}(t_{n,k})+\dfrac{t-t_{n,k}}{s_{n,k}-t_{n,k}}\big[\pi_{i_{0}}(s_{n,k})-\pi_{k}(t_{n,k})\big] & \text{ for } t_{n,k}< t< s_{n,k},\\
\pi_{i_{0}}(t) & \text{ for } t\geq s_{n,k}. 
\end{array} \right.
\end{align*}

\begin{definition}
For $ n\geq 1 $, the map $ f_{n,\alpha}:\Pi^{k}\rightarrow\Pi^{k} $ given by \begin{align*}
f_{n,\alpha}(\pi_{1},\ldots ,\pi_{k}):= \left\{
\begin{array}{rl}
(\pi_{1},\ldots  ,\pi_{k-1},\bar{\pi}_{k}) ~~~& \text{for~} (\pi_{1},\ldots ,\pi_{k})\in \mathcal{A}_{n,\alpha},\\
(\pi_{1},\ldots ,\pi_{k-1},\pi_{k}) ~~~& \text{otherwise} ,
\end{array} \right.
\end{align*}
 is the \textit{$ \alpha $-coalescence map}.
\end{definition}

 We complete the verification of condition ($ \text{I}_{1} $) in the following 3 steps:
\paragraph{Step 1}
We construct paths $ \pi_{1},\ldots ,\pi_{k-1} $ starting from $ \bx_{n}^{1},\ldots ,\bx_{n}^{k-1} $ as before. However, from $ \bx_{n}^{k} $ we construct a path $ \tilde{\pi}_{k} $ using an independent collection  
$ \{U^{'''}_{\bw} :\bw\in\Z^{2}\} $  of i.i.d.  uniform $ (0,1) $ random variables. 
Clearly $\tilde{\pi}_{k} \disteq \pi_{\bx_{n}^{k}}$ and $ \tilde{\pi}_{k} $ is independent of the paths $ \pi_{1},\ldots ,\pi_{k-1} $. Using the same argument as in \cite{Roy2016}, we have 
\begin{equation}\label{step1}
f_{n,\alpha}(\pn_{1},\ldots ,\pn_{k-1},\tilde{\pi}^{(n)}_{k})\Rightarrow (W_{\bx^{1}},\ldots ,W_{\bx^{k}}) \text{ as }n\rightarrow\infty .
\end{equation}
 Indeed, the induction hypothesis and the independence between $ (\pn_{1},\ldots ,\pn_{k-1}) $ and $ \tilde{\pi}^{(n)}_{k} $ imply that $ (\pn_{1},\ldots ,\pn_{k-1},\tilde{\pi}^{(n)}_{k})\Rightarrow (W_{\bx^{1}},\ldots ,W_{\bx^{k-1}},B_{\bx_{k}}) $ as $ n\rightarrow\infty $. Then, using Lemma 5.5 of \cite{Roy2016} and continuous mapping theorem, we obtain \eqref{step1}.

\paragraph{Step 2}
In this step, we construct another path starting from $ \bx_{n}^{k} $ which is not necessarily independent of $ \pi_{1},\ldots ,\pi_{k-1} $. Let $ \{U_{\bw}^{\ast}:\bw\in\Z^{2}\} $ be such that
\begin{equation}
\label{rregion}
U_{\bw}^{\ast}:= \begin{cases}U_{\bw} & \text{ for } \bw\in \bigcup_{i=1}^{k-1}\bigcup_{m=0}^{\infty}V(h_{m}(\bx_{n}^{i}),Y_{m+1}^{\bx_{n}^{i}}),\\
U^{'''}_{\bw} & \text{ otherwise}. 
\end{cases}
\end{equation}
Here
$ \bigcup_{i=1}^{k-1}\bigcup_{m=0}^{\infty}V(h_{m}(\bx_{n}^{i}),Y_{m+1}^{\bx_{n}^{i}}) $
 is the set of all vertices explored to construct $ \pi_{1},\ldots ,\pi_{k-1} $.

Let $\pi_k$ be the path starting from $ \bx_{n}^{k} $ constructed using $ \{U_{\bw}^{\ast}:\bw\in\Z^{2}\} $.  Hence, $ \pi_ {k} $  is independent of $ \pi_1, \ldots, \pi_ {k-1} $ as long as it stays away from the explored region $ \bigcup_{i=1}^{k-1}\bigcup_{m=0}^{\infty}V(h_{m}(\bx_{n}^{i}),Y_{m+1}^{\bx_{n}^{i}}) $  of $ \pi_1, \ldots, \pi_ {k-1} $.
In this step, we want to prove that
\begin{equation}\label{step2}
f_{n,\alpha}(\pn_{1},\ldots ,\pn_{k})\Rightarrow(W_{\bx^{1}},\ldots ,W_{\bx^{k}}) \text{ as }n\rightarrow\infty .
\end{equation}
Let $\pn_{1,k}$ and $\pn_{2,k}$ be such that
$f_{n,\alpha}(\pn_{1},\ldots ,\pn_{k-1},\tilde{\pi}^{(n)}_{k}):=(\pn_{1},\ldots ,\pn_{k-1},\pn_{1,k})$ and 
$f_{n,\alpha}(\pn_{1},\ldots ,\pn_{k-1},\pn_{k}):=(\pn_{1},\ldots ,\pn_{k-1},\pn_{2,k})$.
From \eqref{step1}, we need to show 
\[d_{\Pi}^{(k)}\big((\pn_{1},\ldots ,\pn_{k-1},\pn_{1,k}),(\pn_{1},\ldots ,\pn_{k-1},\pn_{2,k})\big)=d_{\Pi}(\pn_{1,k},\pn_{2,k})\prob 0 \text{ as }n\rightarrow \infty ,\]
i.e., to show \eqref{step2} we need to  prove that for any $ t>0 $,
\begin{equation}\label{step21}
\sup_{0\leq s\leq t}\vert \pn_{1,k}(s)-\pn_{2,k}(s)\vert \prob 0 \text{ as }n\rightarrow \infty .
\end{equation}

Fix $ t>0 $. For $ s>0 $, $1 \leq i \leq k$ and $ j\geq 1 $, let
$ N_{i}(s)=N_{i}(s,n):=\max\{m \geq 0:h_m({\bx_{n}^{i}})(2)\leq  n^{2}\gamma s\} $ and $t_j({\bx_{n}^{i}}):=  2(h_j({\bx_{n}^{i}})(2) - h_{j-1}({\bx_{n}^{i}})(2)) $.
Here we note that for the vertex ${\bx_{n}^{k}}$, we obtain the path $ \{h_j({\bx_{n}^{k}}):j\geq 0\} $ using the collection  $ \{U_{\bw}^{\ast}:\bw\in\Z^{2}\} $.
Taking 
\[W_{n,s}:=\bigcap_{i=1}^{k}\bigcap_{j=1}^{N_{i}(s) +1 }\big\{  t_j({\bx_{n}^{i}}) <\dfrac{\sigma n^{\beta}}{2}\big\},\]
we have
\begin{align*}
\P \{W_{n,s}^{c}\}&\leq\sum_{i=1}^{k}\sum_{j=1}^{-\bx_{n}^{i}(2)+\lfloor n^{2}\gamma s\rfloor +1}\P\big\{ t_j({\bx_{n}^{i}}) \geq \dfrac{\sigma n^{\beta}}{2}\big\}
\\ &\leq k(-\bx_{n}^{1}(2)+\lfloor n^{2}\gamma s\rfloor +1)\P\big\{t_1({\bx_{n}^{1}})\geq \dfrac{\sigma n^{\beta}}{2}\big\}
\rightarrow 0 \text{ as } n \to \infty \text{ for } \beta > 0.
\end{align*}
Note that the occurrence of the event $ W_{n,s} $ implies that, for each $i = 1, \ldots, k$,
 at each jump $h_m({\bx_{n}^{i}})$ comprising the path $\pi_i$ lying below the line $ \{y=n^{2}\gamma s\} $, $\pi_i$ explores an isosceles triangle above it of height at most 
$ \frac{\sigma n^{\beta}}{4} $.
By the definition of $ t_{n,k} $, for all $ 1\leq i\leq k-1 $ and  $ 0\leq s\leq t_{n,k} $, we have
\[\vert\pi^{(n)}_{k}(s)-\pn_{i}(s)\vert\geq n^{-\alpha}.\]
Hence, for $ 0\leq s\leq n^{2}\gamma  t_{n,k} $, we have
\[\min_{1\leq i\leq k-1}\vert\pi_{k}(s)-\pi_{i}(s)\vert\geq \sigma n^{1-\alpha}.\]
Fix $ 0 < \beta < 1-\alpha$. We observe that
\begin{itemize}
\item[(i)] if $ t\leq t_{n,k} $ then
$\min\{\vert\pi_{k}(s)-\pi_{i}(s)\vert : 1\leq i\leq k-1\}
\geq \sigma n^{1-\alpha}$ 
for $ s\leq n^{2}\gamma t $.
Also,  since $ \sigma n^{\beta}<\sigma n^{1-\alpha} $,  on the event $ W_{n,t} $, the path $ \pi_{k} $  stays away from the region $\bigcup_{i=1}^{k-1}\bigcup_{m=0}^{\infty}V(h_{m}(\bx_{n}^{i}),h_{m+1}(\bx_{n}^{i})(2)-h_{m}(\bx_{n}^{i})(2))$ (which follows the definition \eqref{rregion} of $U_{\bw}^{\ast}$) until the end of its first jump after line $ \{y=n^{2}\gamma t\} $
and therefore  $\pi_{k} $
agrees with $ \tilde{\pi}_{k} $  on $ [0,n^{2}\gamma t] $, i.e.  $ \pn_{1,k} $ and $ \pn_{2,k} $ agree on $ [0, t] $.
\item[(ii)] if $ t > t_{n,k} $ then the above argument gives that, on the event $ W_{n,t} $,
 $\pi_{k} $ agrees with $ \tilde{\pi}_{k} $  on $ [0,n^{2}\gamma   t_{n,k} ] $, i.e.  $ \pn_{1,k} $ and $ \pn_{2,k} $ agree on $ [0, t_{n,k}] $. Further, since $ \pi_{k}^{(n)}( t_{n,k})=\tilde{\pi}_{k}^{(n)}( t_{n,k}) $, by the definition of  $ \alpha $-coalescence map we have
$ \pn_{1,k} = \pn_{2,k} $ on $ [ t_{n,k},t] $. Therefore, $ \pn_{1,k} $ and $ \pn_{2,k} $ agree on $ [0, t] $ in this case too.
\end{itemize}
Finally $\P(W_{n,t}) \to 1 $ as $n \to \infty$, so
the statement \eqref{step21} is established.

\paragraph{Step 3}
To complete the verification of condition ($ \text{I}_{1} $), it suffices to show that 
\[d_{\Pi}^{(k)}\big((\pn_{1},\ldots ,\pn_{k-1},\pn_{2,k}),(\pn_{1},\ldots ,\pn_{k-1},\pn_{k})\big)=d_{\Pi}(\pn_{2,k},\pn_{k})\prob 0 \text{ as }n\rightarrow \infty ,\]
i.e., we need  to prove that for any $ t>0 $,
\begin{equation}\label{step31}
\sup_{0\leq s\leq t}\vert \pn_{2,k}(s)-\pn_{k}(s)\vert \prob 0~~~~~\text{as}~n\rightarrow \infty .
\end{equation}
Fix $ t>0 $.  When $ t\leq t_{n,k} $, the two paths $ \pn_{2,k} $ and $ \pn_{k} $ agree on $ [0,t] $. So let  $ t> t_{n,k} $.  Since $ \pn_{2,k} $ and $ \pn_{k} $ agree on $ [0,t_{n,k}] $, we have
$$
\sup_{0\leq s\leq t}\vert \pn_{2,k}(s)-\pn_{k}(s)\vert 
\leq\sup_{ t_{n,k}\leq s\leq t}\vert \pn_{i_{0}}(s)-\pn_{k}(s)\vert .
$$

 We restrict ourselves to the event $ W_{n,t_{n,k}} $. On this event, for the determination of the index  $i_0$ introduced in \eqref{ri0}, we only need to know the configuration in $\{\bw\in\Zz : \bw(2) \leq \lfloor n^{2}\gamma t_{n,k}\rfloor +\lfloor \sigma n^{\beta}/4\rfloor +1\} $. Moreover, our  graph is such that the displacement of the first coordinate during a time $s$ is at most $s$.  So, taking $n$ large  such that $ \lfloor\sigma n^{\beta}/4\rfloor\geq 1 $, we have
\begin{align}\label{a-star}
\big\vert \pi_{i_{0}}(s)-\pi_{k}(s)\big\vert \leq \sigma n^{1-\alpha}+\sigma n^{\beta} 
\text{ for }  n^{2}\gamma t_{n,k} \leq s\leq \lfloor n^{2}\gamma t_{n,k}\rfloor+\lfloor\sigma n^{\beta}/4\rfloor +1.
\end{align}
Taking $ a_{n,k}:=\big( \lfloor n^{2}\gamma t_{n,k}\rfloor+\lfloor\sigma n^{\beta}/4\rfloor +1 \big)/(n^{2}\gamma) $,  we have
\begin{align*}\label{step3(1)}
\sup_{ t_{n,k}\leq s\leq a_{n,k}}\vert \pn_{i_{0}}(s)-\pn_{k}(s)\vert \leq  n^{-\alpha}+ n^{\beta -1} \rightarrow 0~~~~~\text{as}~n\rightarrow \infty .
\end{align*}
Without loss of generality, assume that $ \pi_{i_{0}}(s) \leq  \pi_{k}(s) $ for $ s\geq 0 $. From \eqref{a-star}   we can find $  u_{n} ,v_{n} \in \Z  $   such that $ u_{n} <\pi_{i_{0}}(n^{2}\gamma a_{n,k}) $, $ v_{n}>\pi_{k}(n^{2}\gamma a_{n,k}) $ and $ (v_{n} -u_{n})/n\rightarrow 0 $. Let $ \bu_{n}:=(u_{n},n^{2}\gamma a_{n,k}) $ and $ \bv_{n}:=(v_{n},n^{2}\gamma a_{n,k}) $.
By the non-crossing property of the paths, $ \pn_{i_{0}} $ and $ \pn_{k} $ lie between the paths $ \pn_{\bu_{n}} $ and $ \pn_{\bv_{n}} $ from $ a_{n,k} $ onwards, hence when $ t>a_{n,k} $,
\begin{equation}\label{step39}
\sup_{ a_{n,k}\leq s\leq t}\vert \pn_{i_{0}}(s)-\pn_{k}(s)\vert \leq  \sup_{ a_{n,k}\leq s\leq t}\vert \pn_{\bu_{n}}(s)-\pn_{\bv_{n}}(s)\vert .
\end{equation} 
Now we apply the following proposition:
\begin{proposition}\textup{\textbf{(Proposition 5.3 of \cite{Roy2016})}} 
For $ \bx_n = (x_n , 0) $, $ \by_n = (y_n, 0)$ with $ x_n < 0 < y_n $ and $  (y_n - x_n )/n \to 0 $ as $ n\rightarrow\infty $, we have
\begin{equation*}
 ( \pi_n^{\bx_n}, \pi_n^{\by_n}) \Rightarrow ( B^{{\bf 0}}, B^{{\bf 0}}).
\end{equation*}
\end{proposition}

From this proposition, we conclude that 
\begin{align*}\label{step3(2)}
\sup_{ a_{n,k}\leq s\leq t}\vert \pn_{\bu_{n}}(s)-\pn_{\bv_{n}}(s)\vert \prob 0 ~~~~~\text{as}~n\rightarrow \infty , 
\end{align*}
which along with \eqref{step39} establishes \eqref{step31}.

 \subsection{Verification of condition ($ \text{B} _{1} $)}
 We first estimate the coalescence time of two paths of our model.
\begin{proposition}\label{bound on nu,v}
For $ \bu=(x,0) $ with $ x\in\N $ and $ \bv=(0,0) $, let $ \n $ be defined as
\begin{align*}
\n &=\inf\{n \geq 1 :Z_{n}(\bu ,\bv)=0\},
\end{align*}
where $Z_{n}(\bu ,\bv)$ is as in \eqref{rZ}. Then, for some constant $ C_{11} $, we have
\begin{equation*}
\P \{\n> t\}\leq \dfrac{C_{11}x}{\sqrt{t}} ~~~\text{ for  $ t>0 $}.
\end{equation*}

\end{proposition}

\begin{proof}
For the proof, we use the following result from \cite{Anish2}.

\vspace{3mm}
\begin{theorem} \label{Anish}
Let $ \{V_{n}:n \geq 0\} $ be a discrete-time, positive real valued martingale with respect to a filtration $ \{\mathcal{G}_{n}:n \geq 0\} $. Suppose there exist constants $ C_{12},C_{13}>0 $ so that for any $ n \geq 0 $, on $ \{V_{n}>0\} $,  
\begin{align*}
\E\big[(V_{n+1}-V_{n})^{2}~\big |~ \mathcal{G}_{n}\big]\geq C_{12}~~~\text{and}~~~
\E\big[|V_{n+1}-V_{n}|^{3}~\big |~  \mathcal{G}_{n}\big]\leq C_{13},
\end{align*}
almost surely. If $ \tau^{V}:= \inf\{n \geq 1 : V_{n} = 0\} $, then there is a constant $ C_{14}>0 $ such that for any positive integer $ x $ and positive real valued number $ t $,
\[\P\big\{\tau^{V}> t~\big |~  V_{0}=x\big\}\leq C_{14}x/\sqrt{t}.\]
\end{theorem}
\vspace{3mm}

 Returning to the proof of the proposition, first note that from Theorem \ref{rmart}, $\{Z_{n}(\bu ,\bv):n\geq 0\}$ is a non-negative martingale with respect to the filtration $ \{\mathscr{F}_{T_{n}(\bu ,\bv)}:n\geq 0\} $.
With $\tau_n= \tau_n(\bu ,\bv)$ and $T_n = T_n(\bu ,\bv)$ as defined prior to Lemma \ref{Lemma 2.6 of RSS}, on the event $ \{Z_{n}(\bu,\bv)>0\} $, we have
\begin{align*}
&\E\big[\big(Z_{n+1}(\bu,\bv)-Z_{n}(\bu,\bv)\big)^{2}~\big |~  \mathscr{F}_{T_{n}}\big]\\&~
\geq \P\big\{Z_{n+1}(\bu,\bv)-Z_{n}(\bu,\bv)=1~\big |~  \mathscr{F}_{T_{n}}\big\}\\&~
\geq \P\big\{h_{\tau_{n+1}}(\bu)=h_{\tau_{n}}(\bu)+(1,1),h_{\tau_{n+1}}(\bv)=h_{\tau_{n}}(\bv)+(0,1)~\big |~  \mathscr{F}_{T_{n}}\big\}\\&~
\geq  (1-p)^{4}p^{2}
\end{align*}
almost surely and 
\begin{align}
\label{rsum3}
&\E\Big[\big|Z_{n+1}(\bu,\bv)-Z_{n}(\bu,\bv)\big|^{3}~\big |~  \mathscr{F}_{T_{n}}\big]\nonumber\\&~
=\E\Big[\Big|\sum_{i=\tau_{n}+1}^{\tau_{n+1}}\big (h_{i}(\bu)(1)-h_{i-1}(\bu)(1)\big )-\sum_{i=\tau_{n}+1}^{\tau_{n+1}}\big (h_{i}(\bv)(1)-h_{i-1}(\bv)(1)\big )\Big|^{3}~\big |~  \mathscr{F}_{T_{n}}\Big]\nonumber\\&~
\leq \E\Big[\Big(2\max\Big\{\sum_{i=\tau_{n}+1}^{\tau_{n+1}}\big |h_{i}(\bu)(1)-h_{i-1}(\bu)(1)\big |,\sum_{i=\tau_{n}+1}^{\tau_{n+1}}\big|h_{i}(\bv)(1)-h_{i-1}(\bv)(1)\big |\Big\}\Big)^{3}~\big |~  \mathscr{F}_{T_{n}}\Big]\nonumber\\&~
\leq 16\E\Big[\Big(\sum_{i=\tau_{n}+1}^{\tau_{n+1}}\big |h_{i}(\bu)(1)-h_{i-1}(\bu)(1)\big |\Big)^{3}~\big |~  \mathscr{F}_{T_{n}}\Big]\nonumber\\&~
\leq 16\E\Big[\Big(\sum_{i=1}^{\sigma_{n+1}}\big (h_{\tau_{n}+i}(\bu)(2)-h_{\tau_{n}+i-1}(\bu)(2)\big )\Big)^{3}~\big |~  \mathscr{F}_{T_{n}}\Big]\nonumber\\
& ~ \leq 16  \E\Big[\Big(\sum_{i=1}^{\tau^M} J_i\Big)^3\Big],
\end{align}
 almost surely, where $J_i $ and $ \tau^M $ are as defined in the proof of Proposition \ref{ppppp}, and in the fifth line we have used the fact that $ \vert h_{m}(\bw)(1)-h_{m-1}(\bw)(1)\vert\leq h_{m}(\bw)(2)-h_{m-1}(2) $ for every $\bw \in \Z^2$ and $ m\geq 1 $.
As in the last paragraph of the proof of Proposition \ref{ppppp}
we now obtain that the expectation in \eqref{rsum3} is finite. This completes the proof.
\end{proof}

\begin{proposition}\label{bound on Tnu,v}
Let $ \bu=(x,0) $ with $ x\in\N $ and $ \bv=(0,0) $. Then, for some positive constant $ C_{15} $, we have
\begin{equation*}
\P \big\{T_{\n}(\bu ,\bv)\geq t\big\}\leq \dfrac{C_{15}x}{\sqrt{t}} ~~~\text{ for $ t >0 $}.
\end{equation*}

\end{proposition}
\begin{proof}
 Take $ T_{i}=T_{i}(\bu ,\bv) $ for $ i\geq 0 $. We know that $ \{T_{i}-T_{i-1}:i \geq 1\} $ is a sequence  of independent  random variables and that $ T_{2}-T_{1},\; T_{3}-T_{2},\ldots $ are identically distributed. Since $ T_{1} $ and $ T_{2}-T_{1} $ have exponentially decaying tail probabilities, there is $ s_{0}>0 $ such that for all $ 0\leq s< s_{0} $,
 \[M(s):=\max\big\{\E[\mathrm{e}^{sT_{1}}],\E[\mathrm{e}^{s(T_{2}-T_{1})}]\big\}<\infty .\]
Since $ M $ is an increasing and continuous function on $  (0,s_0) $ with $ M(0)=1 $, we can find $ 0<\alpha<s_0 $ so that $ M(\alpha)\leq\delta $ for some $ \delta \geq 1 $ with  $ \ln \delta <2 $.
 So, for any $ t>0 $, using Proposition \ref{bound on nu,v},
 \begin{align*}
\P \big\{T_{\n}\geq t\big\}
&\leq\P \big\{T_{\n}\geq t,\n \leq \dfrac{\alpha t}{2} \big\}+\P\big \{\n > \dfrac{\alpha t}{2}\big\}\\&
\leq \P \big\{T_{\lfloor \frac{\alpha t}{2}\rfloor }\geq t\big\}+C_{11}x\big(\dfrac{\alpha t}{2}\big)^{-\frac{1}{2}}\\&
= \P \big\{\mathrm{e}^{\alpha T_{\lfloor \frac{\alpha t}{2}\rfloor }}\geq \mathrm{e}^{\alpha t}\big\}+C_{11}x\big(\dfrac{\alpha t}{2}\big)^{-\frac{1}{2}}\\&
\leq \mathrm{e}^{-\alpha t}\E\Big[\exp\Big\{\alpha\sum_{i=1}^{\lfloor \frac{\alpha t}{2}\rfloor} (T_{i}-T_{i-1} )\Big\}\Big]+C_{11}x\big(\dfrac{\alpha t}{2}\big)^{-\frac{1}{2}}\\&
=\mathrm{e}^{-\alpha t}\prod_{i=1}^{\lfloor \frac{\alpha t}{2}\rfloor}\E\big[\mathrm{e}^{\alpha (T_{i}-T_{i-1})}\big]+C_{11}x\big(\dfrac{\alpha t}{2}\big)^{-\frac{1}{2}}\\&
\leq \mathrm{e}^{-\alpha t}(M(\alpha))^{\frac{\alpha t}{2}}+C_{11}x\big(\dfrac{\alpha t}{2}\big)^{-\frac{1}{2}}\\&
\leq \mathrm{e}^{-\alpha t(1-\frac{\ln\delta}{2})}+C_{11}x\big(\dfrac{\alpha t}{2}\big)^{-\frac{1}{2}}\\&
\leq \dfrac{C_{15}x}{\sqrt{t}},
\end{align*}
for a suitable constant $C_{15}$,  where we also take $ \prod_{i=1}^{0}=0 $
\end{proof}

Finally, consider $ t,\varepsilon>0 $, and take $ \bu=(\lceil n\sigma\varepsilon\rceil +3,0) $ and $ \bv=(0,0) $.  Using translation invariance and the non-crossing property of our paths, we have
 \begin{align*}
&\limsup_{n\rightarrow\infty} \sup_{(a,t_{0})\in\mathbb{R}^{2}}\mathbb{P}\big\{\eta_{\bar{\mathcal{X}}_{n}(\sigma ,\gamma)}(t_{0},t;a,a+\varepsilon)\geq 2\big\}\\&
~\leq \limsup_{n\rightarrow\infty}\mathbb{P}\big\{\eta_{\mathcal{X}}(0 ,n^{2}\gamma t ;0,\lceil n\sigma\varepsilon\rceil +3)\geq 2\big\}\\&
~\leq \limsup_{n\rightarrow\infty}\P\big\{\pi_{\bu}(n^{2}\gamma t)\neq\pi_{\bv}(n^{2}\gamma t)\big\}\\
&~\leq  \limsup_{n\rightarrow\infty}\P\big\{T_{\n}(\bu ,\bv)\geq n^{2}\gamma t\big\}\\
&~\leq \dfrac{C_{15}\sigma\varepsilon}{\sqrt{\gamma t}}.
 \end{align*}
Letting $ \varepsilon\downarrow 0 $, condition ($ \text{B}_{1} $) is verified.

 \subsection{Verification of condition ($ \text{B} _{2} $)}
We finish the proof of Theorem \ref{BW} by verifying condition ($ \text{B} _{2} $). We first estimate the expected time of the first collision of two of the three paths starting from the same time level. Let $ x,y,z \in\Z$ with $ x<y<z $. For the processes starting from $ (x,0) $, $ (y,0) $ and $ (z,0) $, and $ l\geq 0 $, let 
\begin{align*}
&r_{l}^{x,y,z}:=\min\big\{h_{l}((x,0))(2),h_{l}((y,0))(2),h_{l}((z,0))(2)\big\},\\&
s_{l}^{x,y,z}:=\max\big\{h_{l}((x,0))(2),h_{l}((y,0))(2),h_{l}((z,0))(2)\big\}.
\end{align*}
Also, taking 
\[ \tau_{l}^{x,y,z}:=\begin{cases} 0 & \text{if } l = 0,\\
\inf\{n>\tau_{l-1}^{x,y,z} :r_{n}^{x,y,z}=s_{n}^{x,y,z}\} & \text{if } l\geq 1, \end{cases}\]
 let  $ T_{l}^{x,y,z}:=h_{\tau_{l}^{x,y,z}}((x,0))(2) $ and, as in Section \ref{s2}, let  $ \mathscr{F}_{t}=\sigma(\{U_{\bw}:\bw\in\Zz,\bw(2)\leq t\}) $ for $ t\geq 0 $.
\begin{proposition}\label{a-proposition1}
For all $ m\geq 1 $ and  positive constants $ C_{16} $ and $ C_{17} $, we have
\begin{align*}
\P\{T_{1}^{x,y,z}\geq m\}\leq C_{16}\exp\{-C_{17}m\}.
\end{align*}
\end{proposition}
\begin{proof}
For $ k\in\{x,y,z\} $, taking $ \bu_{k ,0} :=(k,0) $ and $ \mathtt{t}_{k ,0}:=1 $, we define $ (\bu_{k,n},\mathtt{t}_{k ,n}) $ for $ n\geq 1 $ as follows:
\begin{itemize}
\item[(1)] if $ \mathscr{V} \cap \big([ \bu_{k,n-1}(1) - \mathtt{t}_{k ,n-1}, \bu_{k,n-1}(1) + \mathtt{t}_{k ,n-1}] \times \{n\}\big) = \emptyset $ then take  $ \bu_{k,n} = \bu_{k,n-1}+(0,1)  $ and $  \mathtt{t}_{k ,n} = \mathtt{t}_{k ,n-1} +1 $;
\item[(2)] if $ \mathscr{V} \cap\big( [ \bu_{k,n-1} (1)- \mathtt{t}_{k ,n-1}, \bu_{k,n-1}(1) + \mathtt{t}_{k ,n-1}] \times \{n\}\big) \neq \emptyset $ then take $ \bu_{k,n} = \bv $ and $ \mathtt{t}_{k ,n} = 1 $, where $\bv \in \mathscr{V} \cap \big([\bu_{k,n-1}(1) - \mathtt{t}_{k ,n-1}, \bu_{k,n-1}(1) + \mathtt{t}_{k ,n-1}] \times \{n\}\big) $ is such that $U_{\bv} \leq U_{ \bw}$ for all $\bw \in \mathscr{V} \cap \big([ \bu_{k,n-1}(1) - \mathtt{t}_{k ,n-1}, \bu_{k,n-1}(1) + \mathtt{t}_{k ,n-1}] \times \{n\}\big)$.
\end{itemize}
\vspace*{2mm}
We observe that $ T_{1}^{x,y,z}=\inf\{n\geq 1:\mathtt{t}_{x,n}=\mathtt{t}_{y,n}=\mathtt{t}_{z,n}=1\} $. For each $ n\geq 0 $, let $ \tilde{\mathtt{t}}_{x,n+1}:=\mathbf{1}(\{\bu_{x,n} +(-1,1)\in\mathscr{V}\}) $, $ \tilde{\mathtt{t}}_{y,n+1}:=\mathbf{1}(\{\bu_{y,n} +(0,1)\in\mathscr{V}\}) $ and $ \tilde{\mathtt{t}}_{z,n+1}:=\mathbf{1}(\{\bu_{z,n} +(1,1)\in\mathscr{V}\}) $.
Then $ \{\tilde{\mathtt{t}}_{k,n}:k\in\{x,y,z\},n\geq 1\} $ is a collection of i.i.d. Bernoulli $ (p) $ random variables, and so $ B:=\inf\{n\geq 1:\tilde{\mathtt{t}}_{x,n}\tilde{\mathtt{t}}_{y,n}\tilde{\mathtt{t}}_{z,n}=1\} $ has a geometric $ (p^{3}) $ distribution. Noting that $ T_{1}^{x,y,z}\leq B $, the proof is completed.
\end{proof}
 \begin{remark}\label{a-corollary1}
With  $ B $ as in the proof above we have for each $ n,m\geq 1 $,
\begin{align}\label{B-moment}
\E\big[\big(T_{n}^{x,y,z}-T_{n-1}^{x,y,z}\big)^{m} ~\big |~ \mathscr{F}_{T_{n-1}^{x,y,z}} \big]\leq \E[B^{m}]=C_{18}^{(m)},
\end{align}
where $ C_{18}^{(m)} $ is  some positive constant that only depends on $ m $.
 \end{remark}

For any $ n\geq 0 $, let $ Z_{n}^{x,y}:=h_{\tau_{n}^{x,y,z}}((y,0))(1)-h_{\tau_{n}^{x,y,z}}((x,0))(1) $ and $ Z_{n}^{y,z}:=h_{\tau_{n}^{x,y,z}}((z,0))(1)-h_{\tau_{n}^{x,y,z}}((y,0))(1) $, each of which is non-negative because of the non-crossing property of our paths. Hence, $ \{(Z_{n}^{x,y},Z_{n}^{y,z}):n\geq 0\} $ is a Markov chain on the state space $ \mathcal{S}:=\{(u,v)\in\Zz:u, v\geq 0\} $. Taking $ \mathcal{S}_{0}:=\{(u,v)\in\mathcal{S}:uv= 0\} $, we use the following theorem and estimate the expected value of the first regeneration step in which (at least) one pair of three paths $ \pi_{(x,0)} $, $ \pi_{(y,0)} $ and $ \pi_{(z,0)} $ collide, i.e., $ \mathtt{n}_{x,y,z}:=\inf\{n\geq 1:Z_{n}^{x,y}Z_{n}^{y,z}=0\} $.

\begin{theorem}\textup{\textbf{(Theorem 3.1 of \cite{Anish1})}} \label{CAthm}
Let $ \{W_{j}:j\geq 0\} $ be a Markov chain with countable state space $ \mathcal{M} $, and $ M_{0},M_{1} $ are two disjoint subsets of $ \mathcal{M} $. Suppose there exist a function $ V:\mathcal{M}\rightarrow[0,\infty) $ and two constants $ b\geq 0 $ and $ p_{0}>0 $ such that
\begin{itemize}
\item[(i)] $ \E\big[V(W_{1})-V(W_{0})~\big |~ W_{0}=x\big]\leq -1+b\mathbf{1}(x\in M_{1}) $ for all $ x\in M_{0}^{c} $;
\item[(ii)] $ \P\big\{W_{1}\in M_{0} ~\big |~ W_{0}=x\big\}\geq p_{0} $ for all $ x\in M_{1} $.
\end{itemize}
Then taking $ \varrho(M_0):=\inf\{j\geq 1:W_{j}\in M_0\} $, for $ x\in M_{0}^{c} $ we have 
\[\E\big[\varrho(M_{0}) ~\big |~ W_{0}=x\big]\leq V(x)+\dfrac{b}{p_{0}}.\]
\end{theorem}

\begin{theorem}\label{a-B1-thm}
There exist some positive constants $ C_{19} $ and $ C_{20} $ such that
\[\E[\mathtt{n}_{x,y,z}]\leq C_{19}+C_{20}(y-x)(z-y).\]
\end{theorem}
\begin{proof}
For $ k\in\{x,y,z\} $ and $ n\geq 1 $, take $ I_{k}:=h_{\tau_{1}^{x,y,z}}((k,0))(1)-k $ and $ X_{n}^{(k)}:=X_{n}^{(k,0)} $.  Since  $ X_{n}^{(k)} $'s are symmetric and $ I_{k}=\sum_{n=1}^{\tau_{1}^{x,y,z}}X_{n}^{(k)} $, we have $ \E[I_{k}]=0 $. Also by the structure of our graph $ |I_{k}|\leq T_{1}^{x,y,z} $, so from \eqref{B-moment} in Remark \ref{a-corollary1}, $ \E[|I_{k}|^{m}]\leq C_{18}^{(m)} $ for each $ m\geq 1 $. Using Cauchy-Schwarz inequality we have
\begin{align}\label{a-B1-1}
\E\big[Z_{1}^{x,y}Z_{1}^{y,z}-(y-x)(z-y) \big]&
=\E\big[(I_{y}-I_{x})(I_{z}-I_{y}) \big]\\&
\leq 4C_{18}^{(2)}.\nonumber
\end{align}
We also have $ \E[I_{y}^{2}]\geq \P\{I_{y}=1\}\geq p^{3}(1-p)^{6}=:a $. Now choose $ m_{0}\geq 2 $  large enough such that $ \P\{T_{1}^{x,y,z}\geq m_{0}/2\}\leq a^{2}/(36C_{18}^{(4)}) $. 

Let $ \mathcal{S}_{1}:=\{(u,v)\in\mathcal{S}:1\leq u\wedge v\leq m_{0}\} $ and observe that $ (y-x,z-y)\notin\mathcal{S}_0 $ because $ x<y<z $. Suppose $ (y-x,z-y)\notin\mathcal{S}_1 $, then $ \E\big[I_{x}I_{y}\mathbf{1}(\{T_{1}^{x,y,z}\leq m_{0}/2\})\big]=0 $.  In this case, using Cauchy-Schwarz inequality we have
\begin{align*}
|\E[I_{x}I_{y}]|&\leq \big| \E\big[I_{x}I_{y}\mathbf{1}(\{T_{1}^{x,y,z}\leq m_{0}/2\})\big]\big|+ \E\big[|I_{x}||I_{y}|\mathbf{1}(\{T_{1}^{x,y,z}\geq m_{0}/2\})\big]\leq \dfrac{a}{6}. 
\end{align*}
Similarly we have $ |\E[I_{y}I_{z}]|\leq a/6 $ and $ |\E[I_{x}I_{z}]|\leq a/6 $. This together with \eqref{a-B1-1} we have
\begin{align*}
\E\big[Z_{1}^{x,y}Z_{1}^{y,z}-(y-x)(z-y) \big]\mathbf{1}\big(\{(y-x,z-y)\notin\mathcal{S}_{1}\}\big)\leq -\dfrac{a}{2} .
\end{align*}
So, for all $ x,y,z\in\Z $ with $ x<y<z $,  we have
\begin{align*}
&\E\big[Z_{1}^{x,y}Z_{1}^{y,z}-(y-x)(z-y) \big]
\leq -\dfrac{a}{2}+(4C_{18}^{(2)}+\dfrac{a}{2})\mathbf{1}\big(\{(y-x,z-y)\in\mathcal{S}_{1}\}\big).
\end{align*}

On the other hand, if  $ (y-x,z-y)\in\mathcal{S}_{1} $,  assuming, without loss of generality,  $ y-x\leq z-y $, we have
\begin{align*}
&\P\big\{Z_{1}^{x,y}Z_{1}^{y,z}=0 \big\}\\&~
\geq \P\big\{h_{1}((x,0))=h_{1}((y,0))=(x+\lceil\tfrac{y-x}{2}\rceil,\lceil\tfrac{y-x}{2}\rceil),h_{1}((z,0))=(z,\lceil\tfrac{y-x}{2}\rceil) \big\}\\&~
\geq (1-p)^{3(\lceil\tfrac{y-x}{2}\rceil +1)^{2}-5}p^{2}\geq (1-p)^{3(m_{0}+1)^{2}-5}p^{2}.
\end{align*}
Our result now follows from  Theorem \ref{CAthm} by taking  $ \mathcal{M}=\mathcal{S} $, $ M_{0}=\mathcal{S}_{0} $, $ M_{1}=\mathcal{S}_{1} $ and $ V:\mathcal{S}\rightarrow [0,\infty) $ defined by $ V((u,v))=2uv/a $.
\end{proof}

Now we are ready to estimate the random time $ \nu_{x,y,z}:=\inf\{t\geq 1:\pi_{(x,0)}(t)=\pi_{(y,0)}(t) \text{ or }\pi_{(y,0)}(t)=\pi_{(z,0)}(t)\}  $, which is the first collision time of a pair of three paths $ \pi_{(x,0)} $, $ \pi_{(y,0)} $ and $ \pi_{(z,0)} $.

 \begin{theorem}
 There exist  positive constants $ C_{21} $ and $ C_{22} $ such that
\[\E[\nu_{x,y,z}]\leq C_{21}+C_{22}(y-x)(z-y).\]
 \end{theorem}
 \begin{proof}
Taking $ I_{n}^{x,y,z}:=T_{n}^{x,y,z}-T_{n-1}^{x,y,z} $, we have 
$ \nu_{x,y,z}\leq T_{\mathtt{n}_{x,y,z}}^{x,y,z}=\sum_{n=1}^{\mathtt{n}_{x,y,z}}I_{n}^{x,y,z} $.
From \eqref{B-moment} in Remark \ref{a-corollary1},
 \begin{align*}
 \E\Big[\sum_{n=1}^{\mathtt{n}_{x,y,z}}I_{n}^{x,y,z}\Big]&
 =\E\Big[\sum_{n=1}^{\infty}I_{n}^{x,y,z}\mathbf{1}(\{\mathtt{n}_{x,y,z}\geq n\})\Big]\\&
  =\sum_{n=1}^{\infty}\E\Big[\E\big[I_{n}^{x,y,z}\mathbf{1}(\{\mathtt{n}_{x,y,z}\geq n\}) ~\big |~ \mathscr{F}_{T_{n-1}^{x,y,z}}\big]\Big]\\&
    =\sum_{n=1}^{\infty}\E\Big[\mathbf{1}(\{\mathtt{n}_{x,y,z}\geq n\})\E\big[I_{n}^{x,y,z} ~\big |~ \mathscr{F}_{T_{n-1}^{x,y,z}}\big]\Big]
    \leq C_{18}^{(1)}\E[\mathtt{n}_{x,y,z}],
 \end{align*}
which together with Theorem \ref{a-B1-thm} completes the proof.
 \end{proof}
 
 Finally, let $ n\in\N  $ and  $ t,\varepsilon>0 $.  We have
  \begin{align*}
 \sup_{(a,t_{0})\in\mathbb{R}^{2}}\mathbb{P}\big\{\eta_{\bar{\mathcal{X}}_{n}(\sigma ,\gamma)}(t_{0},t;a,a+\varepsilon)\geq 3\big\}&
\leq \mathbb{P}\big\{\eta_{\mathcal{X}}(0 ,n^{2}\gamma t ;0,\lceil n\sigma\varepsilon\rceil +3)\geq 3\big\}\\&
\leq\P\Big\{\bigcup_{i=0}^{\lceil n\sigma\varepsilon\rceil +1}\big\{\nu_{i,i+1,\lceil n\sigma\varepsilon\rceil +3}\geq n^{2}\gamma t\big\}\Big\}\\&~
\leq\sum_{i=0}^{\lceil n\sigma\varepsilon\rceil +1}\dfrac{C_{21}+C_{22}(\lceil n\sigma\varepsilon\rceil +2-i)}{n^{2}\gamma t}\\&~
\leq( n\sigma\varepsilon +3)\dfrac{C_{21}+C_{22}( n\sigma\varepsilon +3)}{n^{2}\gamma t}\rightarrow \dfrac{C_{22}\sigma^{2}\varepsilon^{2}}{\gamma t}
 \end{align*}
as $ n\rightarrow\infty $, and condition $ (\text{B}_{2}) $ is verified.

\section*{Acknowledgements} We thank the referee for his/her suggestions which led to an improvement of the manuscript and filled a lacuna in Section 4.3.
Azadeh Parvaneh thanks University of Isfahan for supporting her as a PhD student. She also thanks the Indian Statistical Institute, Delhi Centre, for the financial support and the hospitality during her visits to the Institute, and Professor Anish Sarkar for many useful discussions.

\end{document}